\documentclass{amsart}
\usepackage[utf8]{inputenc}
\usepackage{verbatim}
\usepackage[english]{babel}
\usepackage{graphicx}
\usepackage{xcolor}
\usepackage[pdfborder={0 0 0}]{hyperref}
\usepackage{lipsum}
\usepackage{enumerate}
\usepackage{physics}
\usepackage{amsfonts}
\usepackage{mathtools}
\usepackage{amssymb}
\usepackage{amsmath}
\usepackage{amsthm}
\usepackage{bbm}
\usepackage{xurl}
\newcommand{\N}{\mathbb{N}}
\newcommand{\Z}{\mathbb{Z}}
\newcommand{\Q}{\mathbb{Q}}
\newcommand{\R}{\mathbb{R}}
\newcommand{\C}{\mathbb{C}}

\newcommand{\A}{\mathbb{A}}
\DeclareMathOperator{\vol}{vol}

\DeclareMathOperator{\diam}{diam}

\DeclareMathOperator{\triv}{triv}
\DeclareMathOperator{\Ind}{Ind}
\DeclareMathOperator{\cInd}{cInd}
\DeclareMathOperator{\cdim}{cdim}

\DeclareMathOperator{\Hom}{Hom}

\DeclareMathOperator{\Lie}{Lie}
\DeclareMathOperator{\End}{End}
\DeclareMathOperator{\WF}{WF}
\DeclareMathOperator{\proj}{proj}
\newtheorem{theorem}{Theorem}[section]
\newtheorem{lemma}[theorem]{Lemma}
\newtheorem{proposition}[theorem]{Proposition}
\newtheorem{corollary}[theorem]{Corollary}
\theoremstyle{definition}
\newtheorem{definition}[theorem]{Definition}
\theoremstyle{remark}
\newtheorem{remark}[theorem]{Remark}

\title[The canonical dimension of compactly induced representations]{A lower bound for the canonical dimension of compactly induced representations}
\author{Mick Gielen}
    \address[M. Gielen]{Mathematical Institute, University of Oxford, Oxford OX2 6GG, UK}
    \email{mick.gielen@maths.ox.ac.uk}

\begin{document}

\begin{abstract}
    The canonical dimension is an invariant attached to admissible representations of $p$-adic reductive groups, which has only received significant attention in the case of mod-$p$ representations. In the case of complex representations, the canonical dimension is closely related to the wavefront set. We find a new lower bound for the canonical dimension of a general compactly induced representation over an arbitrary coefficient field. This lower bound is uniform in the sense that it only depends on the group and not on the representation itself. In many cases, this provides a lower bound for the canonical dimension of supercuspidal representations and in the complex case we get a lower bound for the corresponding wavefront set. In order to obtain this result, we first generalize a result on the asymptotic growth of the cardinality of balls in the Bruhat-Tits building to the case of exceptional types.
\end{abstract}

\maketitle

\tableofcontents

\section{Introduction}
The canonical dimension (also known in the literature as the \emph{Gelfand-Kirillov dimension}) of an admissible representation of a reductive $p$-adic group is an invariant, which has only properly been studied in the case of mod-$p$ coefficients using ring-theoretic methods. In this paper we study the canonical dimension over an arbitrary field of coefficients. 

The canonical dimension can be described as a means to quantify the size of the (usually infinite dimensional) representations of a $p$-adic group $G$. The idea is as follows. Suppose we are given a smooth representation $(\pi,V)$ of $G$ and a decreasing chain $K_0\supseteq K_1\supseteq\dots$ of open compact subgroups of $G$ which form a neighbourhood basis of the identity in $G$. Then we can consider the increasing chain $V^{K_0}\subseteq V^{K_1}\subseteq\dots$ of subspaces consisting of exactly the vectors fixed by the $K_i$. The smoothness property of $\pi$ guarantees that this chain exhausts all of $V$:
\[
V=\bigcup_{n=0}^\infty V^{K_n}.
\]
If $\pi$ is moreover admissible, then all the $V^{K_n}$ are finite dimensional and we can attempt to quantify the size of $\pi$ by quantifying the growth of the dimensions $\dim(V^{K_n})$. This rate of growth depends on the choice of decreasing chain $K_0\supseteq K_1\supseteq\dots$ and hence we should fix such a choice. A convenient choice is provided by the Moy-Prasad filtrations of the parahoric subgroups of $G$. It turns out that $\dim(V^{K_i})$ grows exponentially in $i$ and the canonical dimension of $\pi$ quantifies this exponential growth.

The main result of this paper is a lower bound for the canonical dimension of a compactly induced representation of a split group (see Theorem \ref{thm:fundamental}). This lower bound is uniform in the sense that it depends on neither the subgroup induced from nor the representation of this subgroup, but only on the group $G$. The main motivation for obtaining this result comes from the case of complex coefficients. In this case there is a closely related invariant called the \emph{wavefront set}. The wavefront set is defined in terms of the \emph{Harish-Chandra-Howe local character expansion}, which expands the distribution character of an admissible complex representation near the unit as a finite linear combination in terms of nilpotent orbits. A lot of work has gone into trying to compute these wavefront sets (\cite{barbaschmoy}, \cite{bco}, \cite{degenerate}, \cite{waldspurger}). However, the local character expansion is hard to compute and hence so is the wavefront set. 

As a corollary of our main result, we obtain, under a mild hypothesis on $p$, a lower bound for the canonical dimension of an admissible, complex, supercuspidal representation (see Theorem \ref{thm:corollary}). Because of the close relation between the wavefront set and the canonical dimension, this implies a lower bound for the wavefront set. There are previously known results that could be interpreted as lower bounds for the wavefront set. For example, it follows from the work \cite{moeglin} by Moeglin that wavefront sets of tempered representations of classical groups consist of nilpotent orbits which are distinguished in the sense that they do not intersect any proper Levi subalgebra. However, as far as we know, there are no previously known quantitative lower bounds for the wavefront set such as the one we find in this paper.

In the case where the representation is moreover of depth-zero, we also prove an upper bound for its canonical dimension, which is sharp (see Corollary \ref{cor:complexUpperBound}). This upper bound coincides with the trivial upper bound that arises when examining the canonical dimension using the Harish-Chandra-Howe local character expansion. Nonetheless, our proof of this upper bound is interesting, because it makes no use whatsoever of the local character expansion. The character expansion is a deep theorem that requires a significant amount of analysis to prove. On the other hand, our methods are elementary, except where we make use of Bruhat-Tits theory. As far as we know, this perspective on the wavefront set from point of view of the canonical dimension and Bruhat-Tits theory is completely new. Not only does this new perspective reprove old results, like the upper bound, it also proves something entirely new, namely the lower bounds find. This suggests that the canonical dimension is worthy of further study.

Outside of the complex case, the implications of our main result are less clear. In the case of mod-$p$ representations, our lower bounds are of no significance, since our methods easily show that in this case no compactly induced representation can be admissible. In the mod-$l$ case we also find lower bounds for the canonical dimension supercuspidal representations under some restrictions to $G$. In this case, there is no wavefront set to which we can apply these results. Therefore it would be extra interesting to investigate which implications our results have in the mod-$l$ case.

In order to compute a lower bound for the canonical dimension, we apply a version of Mackey's theorem to a compactly induced representation to describe the canonical dimension of this representation in terms of the cardinalities of subsets of the Bruhat-Tits building of the group $G$. While these subsets depend on the representation that we compactly induce, we find that independently of this representation they contain certain balls in the Bruhat-Tits building. If the asymptotic cardinalities of these balls are known, we can thus obtain lower bounds for the canonical dimension.

The asymptotic cardinalities of these balls were computed in \cite{gao} for classical types. We extend this result to exceptional types (see Theorem \ref{thm:cardinality}). This is a geometrical result of independent interest that we then use to obtain the desired bounds on the canonical dimension.

In later work, we intend to refine our techniques to allow for a more precise calculation that might compute the canonical dimension exactly. Moreover, it would be interesting to see if our results for supercuspidal representations can be generalized to arbitrary representations. To see why this might be plausible, we note that it is known how the wavefront set behaves with respect to parabolic induction and thus we obtain lower bounds for the wavefront sets of parabolically inductions of supercuspidal representations. Moreover, in this paper we compute the canonical dimension of a parabolically induced representation in terms of the canonical dimension of the original representation (see Theorem \ref{thm:parabolic}). Thus the only obstacle left to finding lower bounds for the canonical dimension of an arbitrary irreducible complex representation is to study the behaviour of the canonical dimension with respect to taking subquotients. This could be a serious obstacle however and currently we do not know of an obvious way to tackle this problem.

\subsection{Structure}
In Section \ref{sec:preliminary} we recall basic preliminaries regarding parahoric subgroups and the Bruhat-Tits building. In particular, we recall the Moy-Prasad filtrations. In Section \ref{sec:building} we generalize the results of \cite{gao} to exceptional types and obtain the necessary understanding of the asymptotic growth of cardinalities of balls in the Bruhat-Tits building.
In Section \ref{sec:cdim} we define the canonical dimension and establish some basic properties. In Section \ref{sec:main} we obtain our main result: lower bounds for the canonical dimensions of compactly induced representations.
In Section \ref{sec:local} we recall the theory necessary to formulate the local character expansion. Then in Section \ref{sec:complexcdim} we use the local character expansion to prove some basic properties of the canonical dimension of complex representations and to relate the canonical dimension to the wavefront set. Finally, in Section \ref{sec:supercuspidal} we obtain upper bounds for the canonical dimensions of depth-zero supercuspidal complex representations.

\subsection{Acknowledgments}
The author would like to thank his advisor Dan Ciubotaru for suggesting this topic as well as the many invaluable discussions that followed. The author would also like to thank Emile Okada for fruitful discussions about this paper. This work was supported by a United Kingdom Research and Innovation (UKRI) Engineering and Physical Sciences Research Council (EPSRC) scholarship.

\subsection{Notation}
Throughout this text, $F$ denotes a $p$-adic field (i.e. a finite extension of $\Q_p$ for some prime $p$) with ring of integers $\mathcal{O}$. We write $\mathfrak{p}$ for the maximal ideal in $\mathcal{O}$ and we fix a uniformizer $\varpi\in\mathfrak{p}$. The residue field $\mathcal{O}/\mathfrak{p}$ is denoted by $k$. It is a field of characteristic $p$ and we denote its cardinality by $q$. The normalized valuation on $F$ is denoted $\omega$ and the associated norm is $|\cdot|$. We normalize this norm such that $|p|=p^{-1}$.

We fix a split, simply connected, absolutely almost simple algebraic group $\mathbb{G}$ defined over $F$ and let $G=\mathbb{G}(F)$ be its group of $F$-rational points. Such groups are classified by Dynkin diagrams. We expect our results to generalize straightforwardly to split reductive groups of any kind, however we will stick to the simply connected, absolutely almost simple case for simplicity. We also fix a maximal torus $T\subseteq G$ which splits over $F$. Write $X^*(T)$ ($X_*(T)$) for the characters (cocharacters) of $T$ and let $\Phi\subseteq X^*(T)$ be the root system of $G$ associated with $T$. Note that our assumptions on $\mathbb{G}$ guarantee that $\Phi$ is an irreducible, reduced root system. Unless otherwise specified, we fix a base $\Pi=\{\alpha_1,\dots,\alpha_d\}$ for $\Phi$, where $d$ is the rank of $G$. This partitions the set of roots $\Phi$ into 
the set $\Phi^+$ of positive roots and the set $\Phi^-$ of negative roots. We write $\alpha_0$ for the longest root corresponding to our base $\Pi$. We denote the Lie algebra $\Lie(\mathbb{G})$ by $\mathfrak{g}$ and refer to it as the Lie algebra of $G$.

We will consider representations over an arbitrary fixed field $\mathcal{F}$. We will write $\Ind$ and $\cInd$ respectively for the induction and compact induction functors.

\section{Parahoric subgroups and the Bruhat-Tits building}\label{sec:preliminary}
In this section, we fix some notation and terminology and recall some basic facts about the Bruhat-Tits building of $G$ (defined in \cite{bruhat1}) and the Moy-Prasad filtrations (defined in \cite{moy-prasad}). 

Given $\alpha\in \Phi$, let $U_\alpha$ be the corresponding root subgroup of $G$, then $U_\alpha$ is isomorphic to the additive group $\mathbb{G}_a$. We fix a \emph{pinning} of $(G,T)$, which means that we fix isomorphisms $u_\alpha: \mathbb{G}_a\to U_\alpha$ for all $\alpha\in\Phi$ such that
\begin{enumerate}[(i)]
    \item for all $\alpha\in \Phi$ there exists a unique isomorphism $\zeta_\alpha$ from $SL_2$ to the subgroup generated by $U_\alpha$ and $U_{-\alpha}$ such that for all $u\in F$
    \[
    \zeta_\alpha\left(\begin{pmatrix}1 & u \\ 0& 1\end{pmatrix}\right) = u_\alpha(u)
    \]
    and
    \[
    \zeta_\alpha\left(\begin{pmatrix}1 & 0 \\ u& 1\end{pmatrix}\right) = u_{-\alpha}(u).
    \]
    \item for all $\alpha,\beta\in \Phi$ with $\beta\neq -\alpha$ there exist $C_{\alpha,\beta,k,l}\in \Z$ such that for all $u,v\in F$ the commutator between $u_\alpha(u)$ and $u_\beta(v)$ satisfies
    \[
    [u_\alpha(u),u_\beta(v)] = \prod_{k,l\in \N_{>0}, k\alpha+l\beta\in\Phi} u_{k\alpha+l\beta}\left(C_{\alpha,\beta,k,l} u^kv^l\right),
    \]
    where we have fixed an ordering of $\Phi$ (upon which the structure constants $C_{\alpha,\beta,k,l}$ will depend).
\end{enumerate}
We define the subgroup $U^+$ (respectively $U^-$) to be the subgroup of $G$ generated by the $U_\alpha$ with $\alpha\in \Phi^+$ (respectively $\alpha\in \Phi^-$). The following proposition is proven in \cite{bruhat1} (see Proposition 6.1.6 in \cite{bruhat1}).
\begin{proposition}\label{prop:bijProd}
    Let $\Phi^+$ and $\Phi^-$ be ordered arbitrarily, then the product maps
    \[
    \prod_{\alpha\in \Phi^{\pm}} U_\alpha\to U^{\pm}
    \]
    are bijective.
\end{proposition}
The pinning gives us filtrations of the root subgroups $U_\alpha$ by setting
\[
U_{\alpha,r} \coloneqq\{u_\alpha(u)\mid u\in F, \,\omega(u)\geq r\}
\]
for all $r\in \R$ and $\alpha\in\Phi$. We also have a filtration of the maximal torus $T$ by setting
\[
T_0\coloneqq \{t\in T\mid \forall\chi\in X^*(T),\, \omega(\chi(t))=0\}
\]
and
\[
T_r\coloneqq \{t\in T_0\mid \forall\chi\in X^*(T),\,\omega(\chi(t)-1)\geq r\}
\]
for all $r\in \R$.

\subsection{The apartment}
Let $N\leq G$ be generated by $T$ and $\zeta_\alpha\left(\begin{pmatrix} 0 & 1 \\ -1 & 0 \end{pmatrix}\right)$ for all $\alpha\in \Phi$. Then $N=Z_G(T)$ and the quotient $N/T$ is isomorphic to the (finite) \emph{Weyl group} $W$ of $G$. We write $V\coloneqq X_*(T)\otimes_\Z \R$. It is clear that there is an action of $W$ on $V$ by linear transformations. This action can be extended to an action of $N$ on $V$ by affine transformations. When viewing $V$ as an affine space over itself, we denote it by $\A$. It is called the \emph{apartment of $G$ corresponding to the torus $T$}. We will find it convenient to fix an origin $o$ for $\A$, which corresponds to the origin of $V$. The action of $N$ on $\A$ is via the quotient $N/T_0$, which is isomorphic to a group $\widehat{W}$ called the \emph{affine Weyl group} of $G$. The subset of $\widehat{W}$ acting via affine reflections gives rise to a set of affine hyperplanes in $\A$ and by taking the closed half-spaces on both sides of all those hyperplanes, we obtain an \emph{affine root system} $\Sigma$ in the sense of \cite{bruhat1}. Given $a\in\Sigma$ the hyperplane $\partial a$ is called a \emph{wall}. Every element $a\in \Sigma$ is of the form
\[
a = \alpha+r\coloneqq \{ x\in\A\mid \alpha(x)+r\geq 0\}
\]
for some $\alpha\in\Phi$ and $r\in\Z$ and every set of this form is in $\Sigma$. Here we are abusing notation by writing $\alpha+r$ both for the affine function $x\mapsto \alpha(x)+r$ on $\A$ and for the half-space it defines. We obtain a surjective map $\Sigma\to\Phi:a\mapsto \hat{a}=\alpha$. The action of $N$ on $\A$ induces an action of $N$ on $\Sigma$. Let $a,b\in \Sigma$, then the walls $\partial a$ and $\partial b$ are called \emph{parallel} if $\hat{a}=\pm\hat{b}$, moreover we say that $\partial a$ is parallel to $\hat{a}$.

An equivalence relation on $\A$ is now obtained by specifying that two points are equivalent when the sets of affine roots they are contained in are the same. The closures of the equivalence classes under this relation are called \emph{facets} and they give rise to a simplicial structure on $\A$, where the equivalence classes are the open cells. The simplices of maximal dimension are called \emph{alcoves} and they are of dimension $d$. Any alcove is a fundamental domain for the action of $N$ on $\A$. Because we have fixed a pinning of $G$ and a base $\Pi$ for our root system, there is a canonical alcove $C$ given by
\[
C\coloneqq \{x\in\A\mid \text{for }1\leq i\leq d, \:\alpha_i(x)\geq 0\text{ and }\alpha_0(x)\leq 1\}.
\]
The vertices of $C$ are $v_0=0$ and $v_1,\dots,v_d$. Here the $v_i$ with $i>0$ are given by $\alpha_j(v_i)=0$ for all $j\neq i,0$ and $\alpha_0(v_i)=1$. We now expand $\alpha_0$ in the basis given by $\Pi$ and find that
\[
\alpha_0 = \sum_{i=1}^d c_i \alpha_i
\]
for some positive integers $c_i$. We define $\omega_i\coloneqq c_i v_i$ and then we note that $\alpha_i(\omega_j)=\delta_{ij}$. Thus the $\omega_i$ form a basis for $V$ dual to $\Pi$ and they are called the \emph{fundamental coweights}.

Associated to $C$ there is a \emph{fundamental Weyl chamber} $C^+\coloneqq \R_{\geq 0}\cdot C$. Because $C$ is a fundamental domain for the action of $N$ on $\A$, every vertex $x$ in $\A$ is $G$-conjugate to exactly one of the $v_i$, which gives rise to a map $\lambda:\A_0\to\{0,\dots,d\}$ given by $x\mapsto i$. Here $\A_0$ is the set of vertices in $\A$. Given a vertex $x\in \A_0$, $\lambda(x)$ is called the \emph{type} of $x$ and the map $\lambda$ is a \emph{typing}.

A vertex $x$ in $\A$ is called \emph{special} if for any $a\in \Sigma$ there is a $b\in\Sigma$ parallel to $a$ such that $x\in\partial b$. Given $x,y\in\A$, a wall $\partial a$ with $a\in\Sigma$ is said to \emph{separate} $x$ and $y$ if $a(x)<0<a(y)$ or $a(y)<0<a(x)$. 

\subsection{The Moy-Prasad filtrations}
The Moy-Prasad filtrations were first defined in \cite{moy-prasad}. In this subsection we briefly recall the necessary definitions and results.

Given $a\in\Sigma$ let $\alpha\in\Phi$ and $r\in \Z$ be such that $a=\alpha+r$. We then define
\[
U_a\coloneqq U_{\alpha,r}.
\]
If $r\in \R$, then we also define 
\[
U_{\alpha,r}\coloneqq U_{\alpha,\lceil r\rceil}.
\]

Let $\widetilde{\Phi}\coloneqq \Phi\cup\{ 0\}$ and let $f:\widetilde{\Phi}\to \R$, then we define $U_f$ to be the subgroup of $G$ generated by the $U_{\alpha,f(\alpha)}$ for all $\alpha\in\Phi$. We also define $P_f$ to be generated by $U_f$ and $T_{f(0)}$. We note that $U_f$ is normalized by $T_0$ and hence $P_f=T_{f(0)}U_f$. We write $N_f\coloneqq U_f\cap N$, $U_f^+\coloneqq U_f\cap U^+$ and $U_f^-\coloneqq U_f\cap U^-$. When dealing with subgroups of the form $P_f$ it is useful to restrict to a certain class of functions $f$.
\begin{definition}
    A function $f:\widetilde{\Phi}\to\R$ is called \emph{concave} if
    \begin{enumerate}[(i)]
        \item for all $\alpha,\beta\in\Phi$ such that $\alpha+\beta\in\Phi$ we have $f(\alpha)+f(\beta)\geq f(\alpha+\beta)$,  
        \item for all $\alpha\in\Phi$ we have $f(\alpha)+f(-\alpha)\geq f(0)$,
        \item $f(0)\geq 0$.
    \end{enumerate}
\end{definition}
\begin{proposition}\label{prop:fprod}
    Let $f:\widetilde{\Phi}\to\R$ be a concave function, then
    \[
    U_f=U_f^-U_f^+N_f.
    \]
    Moreover, for an arbitrary ordering of $\Phi^{\pm}$, the product map
    \[
    \prod_{\alpha\in\Phi^\pm}U_{\alpha,f(\alpha)}\to U_f^\pm
    \]
    is bijective. If in addition we have $f(0)>0$, then $N_f\leq T$ and the product map
    \[
    U_f^-\times U_f^+\times N_f\to U_f
    \]
    is bijective. It follows that in this case the product map
    \[
    U_f^-\times U_f^+ \times T_{f(0)}\to P_f
    \]
    is bijective.
\end{proposition}
\begin{proof}
    This follows from Proposition 6.4.9 in \cite{bruhat1} and Proposition 7.3.12 in \cite{kaletha}.
\end{proof}
\begin{remark}
    Since $U_f$ and $N_f$ are independent of the choice of a base for the root system $\Phi$, we can use the opposite base to find that also
    \[
    U_f=U_f^+U_f^-N_f
    \]
    and similarly we obtain variants of the other assertions.
\end{remark}

The following lemma will be crucial later.
\begin{lemma}\label{lem:crucial}
    Let $f,g:\widetilde{\Phi}\to \R$ be concave functions such that $f(0)=g(0)>0$ and $g\geq f$, then $P_g\leq P_f$ and
    \[
    [P_f:P_g] = \prod_{\alpha\in\Phi} q^{\lceil g(\alpha)\rceil - \lceil f(\alpha)\rceil}
    \]
\end{lemma}
\begin{proof}
    This is Lemma 8.3.1 in \cite{gao}.
\end{proof}

For a bounded subset $\Omega\subseteq \A$ we define a function
\[
f_\Omega:\widetilde{\Phi}\to \R
\]
by $f_\Omega(0)=0$ and for all $\alpha\in\Phi$, $f(\alpha)=\inf\{r\in\R\mid \Omega\subseteq \alpha+r\}$. Then $f_{\Omega}$ is a concave function. If $\Omega=\{x\}$ is a singleton, we will write $f_x$ instead of $f_{\{x\}}$. Note that $f_x(\alpha)=-\alpha(x)$ for all $x\in\A$ and $\alpha\in\Phi$. Moreover, we denote $P_\Omega\coloneqq P_{f_\Omega}$ and $P_x\coloneqq P_{f_x}$. 

If $x\in\A$, then $P_x$ is called a \emph{parahoric subgroup} and it is an open compact subgroup of $G$. Moy and Prasad defined filtrations of the parahoric subgroups by setting
\[
P_{x,r}\coloneqq P_{f_x+r},
\]
where $r\in\R_{\geq 0}$. It is clear that if $s > r$ are nonnegative real numbers, then $P_{x,s}\leq P_{x,r}$.

One nice property of these parahoric subgroups is that some of them allow for a decomposition of the group $G$ as a product of a parabolic part and a compact part. This is also the first time we can recognize the importance of special vertices in the apartment.
\begin{theorem}[Iwasawa decomposition]
    Let $Q\leq G$ be a parabolic subgroup with Levi factor $M\leq Q$ such that $T\leq M$. Let $x\in \A$ be special, then
    \[
    G=QP_{x}.
    \]
\end{theorem}
\begin{proof}
    This follows from \cite{bruhat1} (see Subsection 4.4).
\end{proof}

\begin{definition}
    Let $f:\tilde{\Phi}\to \R$, then we define $f^*:\tilde{\Phi}\to \R$ by
    \[
    f^*(\alpha) = \begin{cases}
        f(\alpha)+1, & \text{if }f(\alpha)\in\Z,\\ \lceil f(\alpha)\rceil, & \text{otherwise.}
    \end{cases}
    \]
    We call $f^*$ the \emph{optimization} of $f$. Let $x\in\A$, then we write $P_{x+}$ instead of $P_{f_x^*}$.
\end{definition}
\begin{remark}
    If $f$ is concave, then so is $f^*$.
\end{remark}
Given $\alpha \in\Phi$, $r\in \R$ and $f:\widetilde{\Phi}\to\R$ we also define \
\[
U_{\alpha,r+}\coloneqq \bigcup_{s>r}U_{\alpha,s}\text{ and }T_{r+}\coloneqq \bigcup_{s>r}T_s.
\]

Clearly $P_{x+}$ is contained in $P_x$ and it is in fact a normal subgroup with quotient $P_x/P_{x+}$ that we denote by $\mathcal{G}_x$. It turns out that $\mathcal{G}_x$ is the group of rational points of some reductive group over the finite field $k$ (see Proposition 6.4.23 in \cite{bruhat1}).

We will need the following technical lemma.
\begin{lemma}\label{lem:intersect}
    Let $x,y\in\A$ and define $f\coloneqq \max(f_x^*,f_y)$, then we have 
    \[
    P_{x+}\cap P_{y}=P_{f}.
    \]
\end{lemma}
\begin{proof}
    The inclusion $P_f\subseteq P_{x+}\cap P_y$ is obvious, therefore we only need to prove the reverse inclusion. Without loss of generality we may assume that the base $\Pi$ for our root system $\Phi$ is chosen such that $y-x$ is in the fundamental Weyl chamber $C^+$. This implies that $\alpha(x)\leq \alpha(y)$ for all $\alpha \in \Phi$ and therefore it follows from Proposition \ref{prop:fprod} that $U_{x+}^+\subseteq U_y^+$.  
    
    Now let $g\in P_{x+}\cap P_{y}$, then by the remark after Proposition \ref{prop:fprod} we have 
    \[
    g=uvn
    \]
    for some $n\in N_{x+}$, $u\in U_{x+}^+$ and $v\in U_{x+}^-$. Now $U_{x+}^+\subseteq U_y^+$ implies that $u^{-1}g\in P_{y}$ and thus 
    \[
    vn= v^\prime u^\prime n^\prime
    \]
    for some $n^\prime\in N_y$, $u^\prime\in U_y^{+}$ and $v^\prime\in U_y^-$. It follows that 
    \[
    n(n^{\prime})^{-1}\in U^-U^+.
    \]
    However, $N\cap U^-U^+$ is trivial (this follows from 6.1.15 \emph{c)} in \cite{bruhat1}), so $n=n^\prime$ and we conclude $n\in N_{f_x^*}\cap N_o=T_{1}\cap N_y=T_{1}\subseteq P_f$. We also find 
    \[
    v^{-1}v^\prime u^\prime=e,
    \]
    where $e$ denotes the identity in $G$. Now using that $U^+\cap U^-$ is trivial we get $u^\prime=e$ and $v=v^\prime$. Therefore $v\in U_y^-\cap U_{x+}^-$ and it follows from Proposition \ref{prop:bijProd} that $U_y^+\cap U_{x+}^-= U_f^-$, therefore $v\in P_f$. We also have that $u\in U_{x+}^+=U_{y}^+\cap U_{x+}^+=U_f^+\subseteq P_f$. Because $u,v,n\in P_f$ it follows that $g\in P_f$. This proves that $P_{x+}\cap P_y\subseteq P_f$ completing the proof.
\end{proof}

\begin{lemma}\label{lem:quotientSize}
    Let $o$ be the origin in $\A$. There exists a constant $\gamma$ such that for every $x\in C^+$ we have
    \[
    \prod_{\alpha\in\Phi^+} q^{\max(\lceil\alpha(x)\rceil-1,0)} \leq [P_o:P_o\cap P_x] \leq \gamma \prod_{\alpha\in\Phi^+} q^{\max(\lceil\alpha(x)\rceil-1,0)}.
    \]
    Moreover, for any $r\in \N$ we have
    \[
    \prod_{\alpha\in\Phi^+} q^{\max(\min(\lceil\alpha(x)\rceil,r)-1,0)} \leq [P_o:P_{o,r}(P_o\cap P_x)] \leq \gamma \prod_{\alpha\in\Phi^+} q^{\max(\min(\lceil\alpha(x)\rceil,r)-1,0)}.
    \]
\end{lemma}
\begin{proof}
    Using the normality of $P_{o+}$ in $P_o$ and the second isomorphism theorem we find
    \begin{align*}
    [P_o:P_x\cap P_o] &= [P_o:P_{o+}(P_x\cap P_o)][P_{o+}(P_x\cap P_o):P_x\cap P_o] \\
    &= [P_o:P_{o+}(P_x\cap P_o)][P_{o+}:(P_x\cap P_o)\cap P_{o+}]\\
    &= [P_o:P_{o+}(P_x\cap P_o)][P_{o+}:P_x\cap P_{o+}].
    \end{align*}
    Similarly, if we fix $r\in \N$ we find
    \begin{align*}
    [P_o:P_{o,r}(P_x\cap P_o)] &=[P_o:P_{o+}(P_x\cap P_o)][P_{o+}(P_x\cap P_o):P_{o,r}(P_x\cap P_o))]\\
    &= [P_o:P_{o+}(P_x\cap P_o)][P_{o+}:(P_{o,r}(P_x\cap P_o))\cap P_{o+}] \\
    &= [P_o:P_{o+}(P_x\cap P_o)][P_{o+}:(P_{o,r}(P_x\cap P_{o+}))].
    \end{align*}
    Now define $f\coloneqq\max(f_o^*,f_x)$ and let $\gamma\coloneqq [P_o:P_{o+}]$, then we can use Lemma \ref{lem:intersect} to see that
    \[
    [P_{o+}:P_f]\leq [P_o:P_x\cap P_o]\leq \gamma [P_{o+}:P_f].
    \]
    The first part of the lemma now follows from Lemma \ref{lem:crucial}. Now note that it is clear that $g\coloneqq\min(r,f)$ defines a concave function and in fact $P_{o,r}(P_x\cap P_{o+})=P_{g}$ and the second part of the lemma follows in the same way.
\end{proof}
\begin{remark}
    This lemma is not particular to the origin. In fact we can choose a different constant $\gamma^\prime$ such that a similar lemma holds for any two points in the apartment $\A$.
\end{remark}

\subsection{The Bruhat-Tits building}
We now define the \emph{Bruhat-Tits building} \\${\mathcal{B}(G)=\mathcal{B}}$ of $G$ as the quotient space $G\times \A/\sim$, where the equivalence relation $\sim$ is defined by
\[
(g,x)\sim (h,y) \text{ if and only if there exists }n\in N\text{ such that }g^{-1}hn\in P_x\text{ and }y=nx.
\]
It is clear that we can identify $\A$ with the points in $\mathcal{B}$ that can be represented by pairs of the form $(e,x)$, where $e$ is the identity in $G$. The simplicial structure on $\A$ gives rise to a simplicial structure on $\mathcal{B}$ and there is an obvious action of $G$ on $\mathcal{B}$ by simplicial automorphisms. We summarize some properties of this action in the following proposition.
\begin{proposition}\label{prop:properties}
    The action of $G$ on $\mathcal{B}$ has the following properties.
    \begin{enumerate}[(i)]
        \item For any $x,y\in\mathcal{B}$, there exists $g\in G$ such that $gx,gy\in\A$.
        \item The stabilizer of a point $x\in\A$ is $P_x$. Moreover, $P_x$ stabilizes the smallest facet in which it is contained.
        \item The (not necessarily pointwise) stabilizer of $\A$ is $N$.
        \item $C$ is a fundamental domain for the action of $N$ on $\A$  and $C^+$ is a fundamental domain for the action of $P_o$ on $\mathcal{B}$.
        \item Let $x,y\in\mathcal{B}$ and let $g,g^\prime\in G$ such that $gx,gy,g^\prime x,g^\prime y\in\A$, then there exists $n\in N$ such that $ngx=g^\prime x$ and $ngy=g^\prime y$.
    \end{enumerate}
\end{proposition}
Using part (i) of the above proposition, we can extend the notion of a special vertex to $\mathcal{B}$. Moreover, from part (ii) it follows that if $x\in \A$ and $g\in G$ is such that $gx\in\A$, then $P_{gx}=gP_xg^{-1}$. This allows us to define $P_x$ for any $x\in\mathcal{B}$ by $P_x=g^{-1}P_{gx}g$ where $g\in G$ is such that $gx\in\mathbb{A}$. We can similarly define the Moy-Prasad filtrations $P_{x,r}$ and $P_x^+$ for arbitrary $x\in\mathcal{B}$. We now give some properties of the Moy-Prasad filtrations.
\begin{proposition}
    Let $x\in\mathcal{B}$, then for any $r\in\R_{\geq 0}$ we have that $P_{x,r}$ is an open compact subgroup of $G$ which is normal in $P_x$ and the collection of all $P_{x,r}$ for varying $r$ gives a neighbourhood basis of the identity in $G$.
\end{proposition}

Given a subset $D$ of $\mathcal{B}$, we write $D_0$ for the set of vertices in $D$. The typing of $\A$ extends to a typing $\lambda:\mathcal{B}_0\to\{0,\dots,d\}$. We note that the orbits of the action of $G$ on $\mathcal{B}_0$ are precisely the sets of vertices in $\mathcal{B}$ of a given type. Given a type $i\in\{0,\dots,d\}$ and a subset $D$ of $\mathcal{B}_0$, we will write $D^i\subseteq D$ for the subset of vertices of type $i$.

The Moy-Prasad filtration subgroups of a Levi subgroup inside of $G$ are essentially obtained from those in $G$.
\begin{lemma}\label{lem:parabolic}
    Let $Q\leq G$ be a parabolic subgroup with Levi factor $M\leq Q$. Then there exists points $x\in \mathcal{B}(G)$ and $\proj(x)\in\mathcal{B}(M)$, where $\mathcal{B}(M)$ is the Bruhat-Tits building of $M$, such that
    \[
    P_{\proj(x),r}=P_{x,r}\cap M
    \]
    for all $r\in\N$. We may choose $x$ to be a special vertex. Moreover, if $T\leq M$ we may in addition choose $x\in \A$.
\end{lemma}
\begin{proof}
    This follows by combining Lemma 8.5.19 in \cite{kaletha} with Proposition 9.8.3 in \cite{kaletha}.
\end{proof}
\begin{remark}
    Note that $M$ need not satisfy the same assumptions we put on $G$, however Bruhat-Tits buildings are defined in greater generality than we have presented here.
\end{remark}

\section{Distances and balls in the Bruhat-Tits building}\label{sec:building}
In this section we introduce a notion of distance between vertices of the Bruhat-Tits building $\mathcal{B}$. This notion is closely related to the \emph{simplicial distance} on the building, obtained by considering the distance in the graph underlying the simplicial structure on $\mathcal{B}$. We provide some estimates for the cardinality of balls of vertices within a certain distance of a given vertex. These cardinalities will later provide us with a lower bound for the canonical dimension.

\begin{definition}
    Let $x,y\in\mathcal{B}_0$ with $x\neq y$, then we define a distance $d(x,y)$ by taking $g\in G$ such that $gx,gy\in\A$. We put $d(x,y)=r$ if the maximal number of parallel walls that separate $gx$ and $gy$ is $r-1$. We also put $d(x,x)=0$. The vertices $x$ and $y$ are called \emph{adjacent} if $d(x,y)=1$.
\end{definition}
\begin{remark}
    It is not yet clear that $d$ is well-defined, however we will prove this in Lemma \ref{lem:metric}.
\end{remark}
\begin{definition}
    Let $x,y\in\mathcal{B}_0$, we write $d^\prime(x,y)$ for the minimal length of a path from $x$ to $y$ in the graph underlying the simplicial structure of $\mathcal{B}$. We call this the \emph{simplicial distance} between $x$ and $y$. 
\end{definition}
\begin{remark}
    It is clear that the simplicial distance is a metric.
\end{remark}
\begin{proposition}\label{prop:distance}
    Let $x,y\in\mathcal{B}_0$, then 
    \[
    d^\prime(x,y)\geq d(x,y).
    \]
    Moreover, we have $d(x,y)=1$ if and only if $d^\prime(x,y)=1$.
\end{proposition}
\begin{proof}
    We may assume $x,y\in\A$. Suppose $d(x,y)=n+1$ and let $\partial a_1,\dots,\partial a_n$ be distinct parallel walls separating $x$ and $y$, then any path from $x$ to $y$ along edges in $\mathcal{B}$, must cross all of these walls and hence must have a vertex on each of these walls. Hence such a path contains at least $n$ vertices unequal to $x$ and $y$ and thus has length at least $n+1$, so $d^\prime(x,y)\geq n+1=d(x,y)$.

    If $d(x,y)>1$, then there is a wall separating $x$ and $y$. Any path from $x$ to $y$ in the graph underlying $\mathcal{B}$ has to contain a vertex on this wall, so $d^\prime(x,y)>1$. Conversely, if $d(x,y)=1$, then there is no wall separating $x$ and $y$, however any vertex unequal to $x$ and $y$ on a shortest path from $x$ to $y$ must lie on a wall separating $x$ and $y$. Hence no such vertex can exist and $d^\prime(x,y)=1$.
\end{proof}
\begin{remark}
    This proposition makes it clear that our notion of distance is very closely related to the simplicial distance in $\mathcal{B}_0$. In fact, it is shown in \cite{gao} that if $G$ is of classical type, then our distance agrees with the simplicial distance. This is not true in general as type $G_2$ provides a counterexample.
\end{remark}
\begin{lemma}\label{lem:metric}
    The function $d:\mathcal{B}_0\times\mathcal{B}_0\to \R_{\geq 0}$ is well-defined, symmetric and positive definite.
\end{lemma}
\begin{proof}
    The function is well defined. To see this, let $x,y\in\mathcal{B}_0$ and let $g,g^\prime\in G$ such that $gx,gy,g^\prime x,g^\prime y\in \A$, then by \ref{prop:properties} there exists $n\in N$ such that $ngx=g^\prime x$ and $ngy=g^\prime y$. The action of $N$ on $\A$ is generated by reflections in the walls in $\A$ and such reflections preserve sets of parallel walls. Hence the maximal number of parallel walls that separate $gx$ and $gy$ is the same as the maximal number of parallel walls that separate $g^\prime x$ and $g^\prime y$.

    It is clear that $d$ is symmetric and that $d(x,y)=0$ only when $x=y$. 
\end{proof}
\begin{lemma}
    Let $x,y,z\in\mathcal{B}_0$ and assume there exists $g\in G$ such that $gx,gy,gz\in\A$, then these point satisfy the triangle inequality:
    \[
    d(x,z)\leq d(x,y)+d(y,z).
    \]
    In particular, if $x,y,z\in\mathcal{B}_0$ are such that at least one pair of them is adjacent, then these points satisfy the triangle inequality.
\end{lemma}
\begin{proof}
    Let $g$ be such that $gx,gy,gz\in \A$. We may assume these three points are distinct. Let $r=d(x,z)$, then there exist parallel walls $\partial a_1,\dots,\partial a_{r-1}$ separating $gx$ and $gz$. Because these walls are parallel at most one of them may contain $gy$. Every other $\partial a_i$ must either separate $gx$ from $gy$ or it separates $gy$ from $gz$. Let $r_1$ be the number of indices $i$ such that $\partial a_i$ separates $gx$ from $gy$ and similarly let $r_2$ be the number of indices $i$ such that $\partial a_i$ separates $gy$ from $gz$, then we have $r_1+r_2\geq r-2$. By definition we also have $d(x,y)\geq r_1+1$ and $d(y,z)\geq r_2+1$ such that
    \[
    d(x,y)+d(y,z)\geq r_1+r_2+2\geq r=d(x,z)
    \]
    completing the proof. 
    
    In general if $x,y,z\in\mathcal{B}_0$ are such that at least one pair of them is adjacent, say $x$ and $y$ are adjacent, then $x$ and $y$ are contained in a common facet and we may choose $x^\prime\in\mathcal{B}$ internal to this common facet. Now by Proposition \ref{prop:properties} there exists $g\in G$ such that $gx^\prime,z\in A$ and thus $gx,gy,gz\in \A$. Therefore the three points $x,y$ and $z$ satisfy the triangle inequality.
\end{proof}
\begin{remark}
    We expect that the function $d$ satisfies the triangle inequality in full generality and thus $d$ is a metric. However, we will only need to make use of the triangle inequality when one pair of points is adjacent, so this lemma suffices for our purposes.
\end{remark}
\begin{definition}
    Let $x\in\mathfrak{B}_0$ be a vertex and let $r\in\Z$, then we define the ball around $x$ of radius $r$ to be
    \[
    B(x,r)\coloneqq \{y\in\mathfrak{B}_0\mid d(x,y)\leq r\}.
    \]
    Note that $B(x,r)=\varnothing$ if $r<0$.
\end{definition}

\begin{lemma}\label{lem:contain}
    Let $r_1,r_2\in\N$ such that $r_1>r_2$ and let $x,y\in \mathcal{B}_0$ be such that $d(x,y)\leq r_1-r_2$. Then we have $P_{x,r_1}\subseteq P_{y,r_2}$. In particular, by taking $r_2=0$, it follows that $P_{x,r_1}$ stabilizes $B(x,r_1-1)$ pointwise.
\end{lemma}
\begin{proof}
    By conjugating we may assume that $x,y\in\A$. It then suffices to show that $f_{x}(\alpha)+r_1\geq f_y(\alpha)+r_2$ for all $\alpha\in\widetilde{\Phi}$. Now clearly $f_{x}(0)+r_1=r_1\geq r_2=f_y(0)$. 
    Suppose that for some $\alpha\in\Phi$, $f_x(\alpha)+r_1<f_{y}(\alpha)+r_2$, then $r_1-r_2<\alpha(x)-\alpha(y)$. Thus for $1\leq m\leq r_1-r_2-1$, we have that $\partial(\alpha-\lceil \alpha(y)\rceil-m)$ defines a wall separating $x$ and $y$. Moreover, if $\alpha(y)\in \Z$, then $\partial(\alpha-\alpha(y)-(r_1-r_2))$ is a wall separating $x$ and $y$ and otherwise the wall $\partial(\alpha-\lceil\alpha(y)\rceil)$ separates $x$ and $y$. This means that there are $r_1-r_2$ parallel walls separating $x$ and $y$, which contradicts $d(x,y)\leq r_1-r_2$. We conclude that $f_{x}(\alpha)+r_1\geq f_y(\alpha)+r_2$ for all $\alpha\in\widetilde{\Phi}$, which finishes the proof.
\end{proof}
The following lemma characterizes the intersection of a ball of a certain radius $r$ centered on the origin $o$ with the fundamental Weyl chamber $C^+\subseteq\A$ and will prove useful later.
\begin{lemma}\label{lem:polytope}
    Let $r\in\N$, then 
    \[
    B(o,r)\cap C^+ = \left(\left(-\alpha_0+r\right)\cap\bigcap_{i=1}^d \alpha_i\right)_0 =(rC)_0\subseteq\A_0.
    \]
\end{lemma}
\begin{proof}
    It is clear that $C^+=\bigcap_{i=1}^d \alpha_i$ and hence it suffices to show that if $x\in C^+_0$, then $x\in B(o,r)$ if and only if $x\in -\alpha_0+r$. Let $x\in C^+_0$ and suppose $x\notin -\alpha_0+r$, then $\alpha_0(x)>r$ and it follows that for all $1\leq i\leq r$ the wall $\partial(-\alpha_0+i)$ separates $o$ and $x$ and hence $d(o,x)\geq r+1$ and thus $x\notin B(o,r)$. 

    Now conversely, let $x\in C^+_0$ and suppose $x\notin B(o,r)$. Then $x$ is separated from $o$ by at least $r$ parallel walls. Let $\alpha\in \Phi^+$ be a root to which these walls are parallel. Let $\partial(\alpha-k)$ be one of these walls for some $k\in\Z$, then because $x\in C^+$ it is clear that $k>0$ and then $\partial(\alpha-i)$ is a separating wall for each $1\leq i\leq k$. Since there must exist at least $r$ such walls, it follows that $\partial(\alpha-r)$ separates $x$ and $o$ and thus $\alpha(x)>r$. Since $x\in C^+$ we have $\alpha_0(x)\geq \alpha(x)>r$ and thus $x\notin -\alpha_0+r$. This proves the first equality and the second equality is obvious.
\end{proof}

\subsection{The cardinality of a ball}
We will find it useful to estimate the number of all vertices of a specific type inside of a ball in terms of the cardinality of a different ball.
\begin{lemma}
    Let $x\in\mathcal{B}_0$ and let $i,j\in\{0,\dots,d\}$ be types, then there exists an injection of $B(x,r)^{i}$ into $B(x,r+1)^{j}$.
\end{lemma}
\begin{proof}
    For each $y\in B(x,r)^i$ pick an alcove $C_y$ containing $y$. It is clear that if $y\neq y^\prime\in B(x,r)^i$, then $C_y\neq C_{y^\prime}$, since alcoves contain exactly one vertex of each type. Given $y\in B(x,r)^{i}$, let $y_j$ be the vertex of $C_y$ of type $j$, then $y$ and $y_j$ are adjacent and thus $d(x,y_j)\leq d(x,y)+1\leq r+1$. We conclude that $y\mapsto y_j$ defines an injection from $B(x,r)^{i}$ into $B(x,r+1)^{j}$.
\end{proof}
\begin{corollary}\label{cor:bound}
Let $r\in \N$ and $0\leq i\leq d$, then
    \[
    |B(x,r)^{i}|\geq \frac{1}{d+1}|B(x,r-1)|.
    \]
\end{corollary}
\begin{proof}
    We have
    \[
    B(x,r-1) = \bigsqcup_{j=0}^d B(x,r-1)^j
    \]
    and by the lemma above this injects into
    \[
    \bigsqcup_{j=0}^d B(x,r)^{i}
    \]
    and hence
    \[
    |B(x,r-1)|\leq (d+1)|B(x,r)^{i}|,
    \]
    from which the result follows.
\end{proof}
The fundamental result which allows us to compute the cardinalities of balls in the Bruhat-Tits building is the following.
\begin{proposition}\label{prop:formula}
There exists a constant $\gamma>0$ such that for all $r\in\N$
    \[
    \sum_{x\in B(o,r)\cap C^+}\prod_{\alpha\in\Phi^+} q^{\max(\lceil\alpha(x)\rceil-1,0)}\leq |B(o,r)| \leq \gamma \sum_{x\in B(o,r)\cap C^+}\prod_{\alpha\in\Phi^+} q^{\max(\lceil\alpha(x)\rceil-1,0)},
    \]
    Moreover, let $r^\prime\in\N$, then
    \begin{align*}
    \sum_{x\in B(o,r)\cap C^+}\prod_{\alpha\in\Phi^+} q^{\max(\min(\lceil\alpha(x)\rceil,r^\prime)-1,0)}&\leq |P_{o,r^\prime}\backslash B(o,r)| 
    \\&\leq \gamma \sum_{x\in B(o,r)\cap C^+}\prod_{\alpha\in\Phi^+} q^{\max(\min(\lceil\alpha(x)\rceil,r^\prime)-1,0)}
    \end{align*}
\end{proposition}
\begin{proof}
    The action of $G$ on $\mathcal{B}_0$ preserves distances and because $P_o$ stabilizes $o$ we thus have $P_oB(o,r)=B(o,r)$. Now using Proposition \ref{prop:properties} it follows that
    \[
    B(o,r) = \bigsqcup_{x\in B(o,r)\cap C^+}P_ox.
    \]
    By the orbit-stabilizer theorem we can conclude
    \[
    |B(o,r)| = \sum_{x\in B(o,r)\cap C^+} [P_o:P_o\cap P_x].
    \]
    Furthermore, using the normality of $P_{o,r^\prime}$ in $P_o$ it is easy to see that $|P_{o,r^\prime}\backslash P_ox| = [P_o:P_{o,r^\prime}(P_o\cap P_x)]$ for all $x\in\mathcal{B}_0$. Therefore
    \[
    |P_{o,r^\prime}\backslash B(o,r)| = \sum_{x\in B(o,r)\cap C^+} [P_o:P_{o,r^\prime}(P_o\cap P_x)].
    \]
    The proposition now follows from Lemma \ref{lem:quotientSize}.
\end{proof}
\begin{remark}
    In \cite{gao} the cardinality of balls is explicitly computed using the simplicial distance, as opposed to our distance function (see Theorem 8.1 in \cite{gao}). Our proof is essentially the same, except we forego an explicit computation in favour of an approximate one. This will be sufficient for our purposes and will save us some work.
\end{remark}
\begin{corollary}\label{cor:ballLowBound}
    Let $r\in\N$, then
    \[
    |B(o,r)|\geq q^{-|\Phi^+|}\max_{x\in B(o,r)\cap C^+} q^{2\rho(x)},
    \]
    where 
    \[
    2\rho = \sum_{\alpha\in \Phi^+}\alpha
    \]
    is the sum of all positive roots.
\end{corollary}
\begin{proof}
    Using the previous proposition we find
    \begin{align*}
    |B(o,r)|&\geq \sum_{x\in B(o,r)\cap C^+}\prod_{\alpha\in\Phi^+} q^{\max(\lceil\alpha(x)\rceil-1,0)}\\
    &\geq \max_{x\in B(o,r)\cap C^+}\prod_{\alpha\in\Phi^+} q^{\max(\lceil\alpha(x)\rceil-1,0)}\\
    &\geq \max_{x\in B(o,r)\cap C^+}\prod_{\alpha\in\Phi^+} q^{\lceil\alpha(x)\rceil-1}\\
    &\geq \max_{x\in B(o,r)\cap C^+}\prod_{\alpha\in\Phi^+} q^{\alpha(x)-1}\\
    &=q^{-|\Phi^+|}\max_{x\in B(o,r)\cap C^+} q^{2\rho(x)}.
    \end{align*}
\end{proof}
\begin{lemma}\label{lem:latticeCount}
    There exists a constant $\Gamma>0$ such that $|B(o,r)\cap C^+|\leq \Gamma r^d$ for all $r\in \N$.
\end{lemma}
\begin{proof}
    By Lemma \ref{lem:polytope} we have $B(o,r)\cap C^+=(rC)_0$. We fix an isomorphism between $\A$ and $\R^d$ which allows us to define a notion of distance and volume on $\A$. Now if $X\subseteq \A$, let us write $\diam(X)$ for the diameter of $X$ and let us write $\vol(X)$ for the volume of $X$ whenever this makes sense. Because the vertices in $\A$ form a regular lattice there exists $\varepsilon>0$ such that no two balls of radius $\varepsilon$ centered at two distinct vertices in $\A$ intersect. Let us write 
    \[
    rC+\varepsilon\coloneqq \bigcup_{x\in rC} b(x,\varepsilon),
    \]
    where $b(x,\varepsilon)$ is the Euclidean ball centered at $x$ with radius $\varepsilon$. Now clearly
    \[
    \bigsqcup_{x\in (rC)_0} b(x,\varepsilon) \subseteq rC+\varepsilon
    \]
    and thus 
    \[
    \vol\left(\bigsqcup_{x\in (rC)_0} b(x,\varepsilon)\right)\leq \vol(rC+\varepsilon).
    \]
    Clearly
    \[
    \vol(rC+\varepsilon)\leq c_d \diam(rC+\varepsilon)^d = c_d(r\diam(C)+2\varepsilon)^d
    \]
    where $c_d$ is a constant depending only on the dimension $d$ of $\A$. Moreover,
    \[
    \vol\left(\bigsqcup_{x\in (rC)_0} b(x,\varepsilon)\right)=c_d|(rC)_0|(2\varepsilon)^d.
    \]
    It follows that
    \[
    |(rC)_0|\leq \left(\frac{r\diam(C)+2\varepsilon}{2\varepsilon}\right)^d
    \]
    and the lemma now follows easily.
\end{proof}
\begin{corollary}\label{cor:upBound}
    Let $r,r^\prime\in\N$, then
    \[
    |P_{o,r^\prime}\backslash B(o,r)| \leq \gamma\Gamma r^d q^{r^\prime|\Phi^+|}.
    \]
\end{corollary}
\begin{proof}
    By Proposition \ref{prop:formula} we have
    \[
    |P_{o,r^\prime}\backslash B(o,r)| \leq \gamma \sum_{x\in B(o,r)\cap C^+}\prod_{\alpha\in\Phi^+} q^{\max(\min(\lceil\alpha(x)\rceil,r^\prime)-1,0)}.
    \]
    Now clearly
    \[
    q^{\max(\min(\lceil\alpha(x)\rceil,r^\prime)-1,0)}\leq q^{r^\prime}
    \]
    and thus
    \begin{align*}
    \gamma \sum_{x\in B(o,r)\cap C^+}\prod_{\alpha\in\Phi^+} q^{\max(\min(\lceil\alpha(x)\rceil,r^\prime)-1,0)}&\leq \gamma \sum_{x\in B(o,r)\cap C^+}q^{r^\prime|\Phi^+|}\\
    &=\gamma |B(o,r)\cap C^+| q^{r^\prime|\Phi^+|}.
    \end{align*}
    The corollary now follows by applying Lemma \ref{lem:latticeCount}.
\end{proof}
Recall that $2\rho$ is the sum of the positive roots $\Phi^+$. We can expand $2\rho$ in terms of the basis $\Pi$ to obtain
\[
2\rho= \sum_{i=1}^d c^\prime_i \alpha_i
\]
for some positive integers $c^\prime_i$. Recall also that we have previously similarly defined the $c_i$ to be the coefficients of $\alpha_0$ with respect to $\Pi$.
\begin{corollary}
    If we write $D\coloneqq \max_{i}\frac{c^\prime_i}{c_i}$, then for all $r\in \N$ we have
    \[
    |B(o,r)|\geq q^{rD-|\Phi^+|}.
    \]
\end{corollary}
\begin{proof}
    Let us write
    \[
    D(r)\coloneqq \max\{2\rho(x)\mid x\in B(o,r)\cap C^+\},
    \]
    then it follows from Corollary \ref{cor:ballLowBound} that
    \[
    |B(o,r)|\geq q^{D(r)-|\Phi^+|}.
    \]
    Thus all that is left to show is that $D(r)=rD$. First we use Lemma \ref{lem:polytope} to see that 
    \[
    B(o,r)\cap C^+=\left(\left(-\alpha_0+r\right)\cap\bigcap_{i=1}^d \alpha_i\right)_0.
    \]
    Now we note that $r\cdot C=\left(-\alpha_0+r\right)\cap\bigcap_{i=1}^d \alpha_i$ is a bounded convex polytope. It is a basic fact that the maximum of a linear functional over a bounded convex polytope is attained at one of its vertices. Thus the maximum of $2\rho$ over $r\cdot C$ is attained at one of the vertices of this polytope. However, this polytope is defined in terms of walls in $\A$ and hence the vertices of this polytope are in fact vertices of $\mathcal{B}$ and thus the maximum of $2\rho$ over $r\cdot C$ is attained at some point in $(r\cdot C)_0$. It follows that
    \[
    D(r) = \max\{2\rho(y)\mid y\in r\cdot C\}
    \]
    from which it is clear that
    \[
    D(r) = r\max\{2\rho(y)\mid y\in C\}.
    \]
    Again using that this maximum must be attained at one of the vertices of $C$ we find
    \[
    D(r) = r\max_{0\leq i\leq d} 2\rho(v_i) = r\max_{0\leq i\leq d} 2\rho\left(\frac{1}{c_i}\omega_i\right)=r\max_{0\leq i\leq d} \frac{c_i^\prime}{c_i}=rD,
    \]
    which completes the proof.
\end{proof}
\begin{corollary}\label{cor:ballCard}
    We have
    \[
    \log_q(|B(o,\bullet)|)\geq D.
    \]
\end{corollary}
\begin{proof}
    This is clear from the previous corollary.
\end{proof}
\begin{remark}
    In fact, it is true that $\log_q(|B(o,\bullet)|) =D$, but we shall not need this result.
\end{remark}
\begin{theorem}\label{thm:cardinality}
    Depending on the type of the root system $\Phi$, we have the following lower bounds for the asymptotic behaviour of the cardinality of balls in the Bruhat-Tits building of $G$:
    \[
        \begin{tabular}{c|c|c|c|c|c|c|}
           \cline{2-7}
            & $A_{2n}$ & $A_{2n+1}$ & $B_3$ & $B_{d\geq 4}$ & $C_d$ & $D_{d\geq 4}$\\ \hline
            \multicolumn{1}{|c|}{$\log_q(|B(o,\bullet)|)\geq$} & $n(n+1)$ & $(n+1)^2$ & $5$ & $\frac{d^2}{2}$ & $\frac{d(d+1)}{2}$ & $\frac{d(d-1)}{2}$ \\\hline
        \end{tabular}
    \]
    \[
        \begin{tabular}{c|c|c|c|c|c|}
           \cline{2-6}
            & $E_6$ & $E_7$ & $E_8$ & $F_4$ & $G_2$\\ \hline
            \multicolumn{1}{|c|}{$\log_q(|B(o,\bullet)|)\geq$} & $16$ & $27$ & $46$ & $11$ & $\frac{10}{3}$ \\\hline
        \end{tabular},
    \]
    where in type $A$ we have $d=2n$ or $d=2n+1$ for some $n\in\N$ depending on the parity of $d$.
\end{theorem}
\begin{proof}
By Corollary \ref{cor:ballCard} it suffices to compute $D=\max_i\frac{c_i^\prime}{c_i}$, depending on the type of $\Phi$. In \cite{bourbaki} we can find $c_i$ and $c_i^\prime$ for all types.

    \vspace{\baselineskip}
    
    {\bf Type $\mathbf{A_d}$.} For all $i$ we have 
    \[
    c_i=1\text{ and }c_i^\prime = d(d-i+1).
    \]
    It follows that $D=\frac{n}{2}(\frac{n}{2}+1)$ for $d$ even and $D=(\frac{n+1}{2})^2$ if $d$ is odd.
    
    \vspace{\baselineskip}
    
    {\bf Type $\mathbf{B_d}$.} For a given $i$ we have 
    \[
    c_i=\begin{cases}1 &\text{if }i=1,\\ 2 &\text{otherwise}\end{cases}\text{ and }c_i^\prime = i(2d-i).
    \]
    It follows that $D=\max\{2n-1,\frac{n^2}{2}\}$ and thus $D=5$ if $n=3$ and $D=\frac{n^2}{2}$ if $n>3$.
    
    \vspace{\baselineskip}
    
    {\bf Type $\mathbf{C_d}$.} We have 
    \[
    c_i=\begin{cases}1 &\text{if }i=d,\\ 2 &\text{otherwise}\end{cases}\text{ and }c_i^\prime = \begin{cases}\frac{d(d+1)}{2} &\text{if }i=d,\\ i(2d-i+1) &\text{otherwise}\end{cases}.
    \]
    It follows that $D=\frac{d(d+1)}{2}$.
    
    \vspace{\baselineskip}
    
    {\bf Type $\mathbf{D_d}$.} We have 
    \[
    c_i=\begin{cases}1 &\text{if }i=1,d-1,d,\\ 2 &\text{otherwise}\end{cases}\text{ and }c_i^\prime = \begin{cases}2(d-1) &\text{if }i=1,\\ \frac{d(d-1)}{2} &\text{if }i=d-1,d \\ 2\left(id-\frac{i(i+1)}{2}\right)& \text{otherwise}\end{cases}.
    \]
    It follows that $D=\max\left\{2(d-1),\frac{d(d-1)}{2},d^2-\frac{d(d+1)}{2}\right\}=\frac{d(d-1)}{2}$, when $d\geq 4$.
    
    \vspace{\baselineskip}
    
    {\bf Type $\mathbf{E_6}$.} We have 
    \[
    (c_1,\dots,c_6)=(1,2,2,3,2,1)\text{ and }(c_1^\prime,\dots,c_6^\prime)=(16,22,30,42,30,16)
    \]
    and thus $D=16$.

    \vspace{\baselineskip}
    
    {\bf Type $\mathbf{E_7}$.} We have 
    \[
    (c_1,\dots,c_7)=(2,2,3,4,3,2,1)\text{ and }(c_1^\prime,\dots,c_7^\prime)=(34,49,66,96,75,52,27)
    \]
    and thus $D=27$.

    \vspace{\baselineskip}
    
    {\bf Type $\mathbf{E_8}$.} We have 
    \begin{align*}
    &(c_1,\dots,c_8)=(2,3,4,6,5,4,3,2)\text{ and }\\&(c_1^\prime,\dots,c_8^\prime)=(92,136,182,270,220,168,114,58)
    \end{align*}
    and thus $D=46$.

    \vspace{\baselineskip}
    
    {\bf Type $\mathbf{F_4}$.} We have 
    \[
    (c_1,c_2,c_3,c_4)=(2,3,4,2)\text{ and }(c_1^\prime,c_2^\prime,c_3^\prime,c_4^\prime)=(16,30,42,22)
    \]
    and thus $D=11$.

    \vspace{\baselineskip}
    
    {\bf Type $\mathbf{G_2}$.} We have 
    \[
    (c_1,c_2)=(3,2)\text{ and }(c_1^\prime,c_2^\prime)=(10,6)
    \]
    and thus $D=\frac{10}{3}$.
\end{proof}
\begin{remark}
    These lower bounds are in fact equalities.
\end{remark}
\begin{remark}
    For classical types these asymptotics were already computed in \cite{gao}. Note that in \cite{gao} they use the simplicial distance function $d^\prime$ instead of $d$. However these different distance functions give the same asymptotics.
\end{remark}

\section{Definition and basic properties of the canonical dimension}\label{sec:cdim}
In this section we will define the canonical dimension of an admissible representation and we show how it behaves with respect to parabolic induction.

As mentioned in the introduction, the canonical dimension quantifies some exponential growth. To quantify the growth of a function that grows exponentially, we introduce the following notation.
\begin{definition}
    Let $f:\N \to \R_{\geq 0}$, then we define
    \[
    \log_q(f)\coloneqq \limsup_{n\to\infty} \frac{\log_q(f(n))}{n}.
    \]
    Here we define $\log_q(0)=-\infty$.
\end{definition}

\begin{definition}
Let $(\pi,V)$ be an admissible representation of $G$ and let $x\in\mathcal{B}$, then the \emph{canonical dimension} $\cdim(\pi)$ of $\pi$ is defined by
\[
\cdim(\pi)\coloneqq \log_q(\dim(V^{P_{x,\bullet}})).
\]
\end{definition}
\begin{remark}
Note that $\cdim(\pi)\geq 0$ because $\dim(V^{P_{x,n}})>0$ for $n$ sufficiently large. 
\end{remark}
\begin{remark}
A priori, the canonical dimension of an admissible representation could be infinite, however we will see later that this does not happen when the representation has complex coefficients and is finitely generated.
\end{remark}
\begin{remark}
The canonical dimension turns out to be independent of the choice of $x\in\mathcal{B}$, hence why it is suppressed in the notation. In fact, any choice of a chain $K_0\supseteq K_1\supseteq\dots$ of open compact subgroups of $G$ such that their inverse images under the exponential map form a decreasing chain of lattices $L_0\supseteq L_1\supseteq\dots$ with $L_{i+1}=\varpi L_i$ for all $i$, will yield the same canonical dimension. However we find it convenient to restrict to Moy-Prasad filtration subgroups.
\end{remark}
\begin{remark} 
The choice of using the logarithm with base $q$ is a convenient normalization. The use of this normalization is justified by Corollary \ref{cor:canDim}.
\end{remark}
\begin{remark}
    The canonical dimension is closely related to the Gelfand-Kirillov dimension for modules over an algebra, however the Gelfand-Kirillov dimension measures polynomial growth, while the canonical dimension measures exponential growth. For this reason we feel it is warranted to distinguish our canonical dimension from the well-known notion of Gelfand-Kirillov dimension. This convention is in line with \cite{james}, where the canonical dimension is studied for representations over a field of positive characteristic. However, elsewhere in the literature our notion of canonical dimension is also simply referred to as the Gelfand-Kirillov dimension.
\end{remark}
\begin{proposition}
    Let $(\pi,V)$ be an admissible representation of $G$, then the canonical dimension $\cdim(\pi)$ of $\pi$ is independent of the point $x\in\mathcal{B}$ used to define it.
\end{proposition}
\begin{proof}
    Let $x,y\in\mathcal{B}$ and let us denote by $\cdim_x(\pi)$ and $\cdim_y(\pi)$ the canonical dimensions of $\pi$ defined with respect to $x$ and $y$ respectively. We need to show that $\cdim_x(\pi)=\cdim_y(\pi)$. It suffices to show this in the case where $y=o$, so we will assume $y=o$. It follows from Proposition \ref{prop:properties} that there exists $g\in G$ such that $gx\in C$. Now we note that
    \[
    gP_{x,r}g^{-1}=P_{gx,r}
    \]
    for all $r\in \R_{\geq 0}$ and we also have
    \[
    gV^{P_{x,r}}=V^{gP_{x,r}g^{-1}}
    \]
    and thus
    \[
    \dim(V^{P_{x,r}})=\dim(gV^{P_{x,r}})=\dim(V^{P_{gx,r}})
    \]
    for all $r\in \R_{\geq 0}$. It is now clear from the definition that $\cdim_x(\pi)=\cdim_{gx}(\pi)$, so we may assume that $x\in C$. 

    Let $x\in C$, then we wish to show that $\cdim_x(\pi)=\cdim_o(\pi)$. Now let $r\geq 1$, then from the definition of $C$ it is clear that 
    \[
    f_x+r-1\leq f_o+r\leq f_{x}+r+1.
    \]
    It follows that 
    \[
    P_{x,r-1}\supseteq P_{o,r}\supseteq P_{x,r+1}
    \]
    and thus
    \[
    V^{P_{x,r-1}}\subseteq V^{P_{o,r}}\subseteq V^{P_{x,r+1}}.
    \]
    It now follows from the definition of the canonical dimension that $\cdim_x(\pi)=\cdim_o(\pi)$.
\end{proof}
\begin{theorem}\label{thm:parabolic}
    Let $Q\leq G$ be a parabolic subgroup with Levi factor $M$ and unipotent radical $U$. Let $(\pi,V_\pi)$ be an admissible representation of $M$, then 
    \[
    \cdim\left(\Ind_Q^G(\pi)\right) = \cdim(\pi) + \dim(U).
    \]
\end{theorem}
\begin{proof}
    Let $x\in\mathcal{B}(G)$ be as in Lemma \ref{lem:parabolic} with associated $\proj(x)\in \mathcal{B}(M)$. We may assume that $T\leq M$ and thus we can choose $x\in\A$. We may furthermore assume that $x$ is a special vertex and that $U^+\leq Q$. These assumptions imply that $Q$ is a standard parabolic given by a subset $\Pi_M\subseteq \Pi$. Then if we let $\Phi_M\subseteq \Phi$ be the set of all roots that are integral linear combinations of elements of $\Pi_M$ we have that $M$ is generated by the root subgroups $U_\alpha$ such that $\alpha\in \Phi_M$. Moreover, $U$ is generated by the root subgroups $U_\alpha$ such that $\alpha\in\Phi^+\setminus \Phi_M$.
    
    Now we fix $r>0$ and we write $V$ for the representation space of $\Ind_Q^G(\pi)$. Then we apply Corollary \ref{cor:mackey} to obtain
    \[
    V^{P_{x,r}}\cong \prod_{QgP_{x,r}\in Q\backslash G/P_{x,r}} V_\pi^{gP_{x,r}g^{-1}\cap Q}
    \]
    where we interpret $V_\pi$ as a representation of $Q$ via the quotient map $Q\twoheadrightarrow M$. Now using the Iwasawa decomposition $G=QP_x$, we see that we can find a complete set of representatives $\Lambda\subseteq P_x$ for the double cosets in $Q\backslash G/P_{x,r}$. We have that if $g\in\Lambda$, then $gP_{x,r}g^{-1}=P_{gx,r}=P_{x,r}$.

    Next we claim that the image of $P_{x,r}\cap Q$ under the quotient map $Q\twoheadrightarrow M$ is $P_{x,r}\cap M$. To show this, it suffices to show that $P_{x,r}\cap Q=(P_{x,r}\cap M)(P_{x,r}\cap U)$ which follows from Proposition \ref{prop:fprod} and the fact that $r>0$.

    Now we have that $P_{x,r}\cap M=P_{\proj(x),r}$ and thus
    \[
    V^{P_{x,r}}\cong \prod_{QgP_{x,r}\in Q\backslash G/P_{x,r}} V_\pi^{P_{\proj(x),r}}.
    \]
    It follows that
    \[
    \dim\left(V^{P_{x,r}}\right)=\dim\left(V_\pi^{P_{\proj(x),r}}\right)\cdot |Q\backslash G/P_{x,r}|.
    \]
    Because of the Iwasawa decomposition it is clear that $|Q\backslash G/P_{x,r}|=|(Q\cap P_x)\backslash P_x/P_{x,r}|$. Using that $P_{x,r}$ is normal in $P_x$ we deduce 
    \[
    |(Q\cap P_x)\backslash P_x/P_{x,r}|=|((Q\cap P_x)P_{x,r})\backslash P_x|.
    \]
    Next we have the following inequalities
    \[
    |((Q\cap P_{x,1})P_{x,r})\backslash P_{x,1}|\leq|((Q\cap P_x)P_{x,r})\backslash P_x|\leq |((Q\cap P_{x,1})P_{x,r})\backslash P_{x,1}|\cdot |P_{x,1}\backslash P_x|.
    \]
    This first inequality is obvious and the second follows from the fact that we have a surjective map
    \[
    ((Q\cap P_{x,1})P_{x,r})\backslash P_{x,1}\times P_{x,1}\backslash P_x\twoheadrightarrow ((Q\cap P_x)P_{x,r})\backslash P_x
    \]
    given by
    \[
    (((Q\cap P_{x,1})P_{x,r})g,P_{x,1}h)\mapsto ((Q\cap P_x)P_{x,r}) gh.
    \]
    These inequalities allow us to conclude that
    \[
    \log_q(|Q\backslash G/P_{x,\bullet}|) = \log_q(|((Q\cap P_{x,1})P_{x,\bullet})\backslash P_{x,1}|).
    \]
    Therefore we have
    \begin{align*}
    \cdim(\Ind_Q^G(\pi))&=\log_q\left(\dim\left(V^{P_{x,\bullet}}\right)\right)\\
    &=\log_q\left(\dim\left(V_\pi^{P_{\proj(x),\bullet}}\right)\right)+\log_q(|((Q\cap P_{x,1})P_{x,\bullet})\backslash P_{x,1}|)\\
    &=\cdim(\pi)+\log_q(|((Q\cap P_{x,1})P_{x,\bullet})\backslash P_{x,1}|).
    \end{align*}
    
    To complete the proof we note that $(Q\cap P_{x,1})P_{x,\bullet}=P_f$, where $f$ is given by
    \[
    f(\alpha) = \begin{cases}
        f_x(\alpha)+1 & \text{if }\alpha\in\Phi_M\cup \Phi^+\cup\{0\},\\
        f_x(\alpha)+r& \text{if }\alpha\in\Phi^-\setminus \Phi_M.
    \end{cases}
    \]
    By Lemma \ref{lem:crucial} it follows that
    \[
    |((Q\cap P_{x,1})P_{x,r})\backslash P_{x,1}|=q^{(r-1)|\Phi^-\setminus \Phi_M|}.
    \]
    Observing that $|\Phi^-\setminus \Phi_M|=|\Phi^+\setminus \Phi_M|=\dim(U)$ it is now clear that
    \[
    \cdim\left(\Ind_Q^G(\pi)\right) = \cdim(\pi) + \dim(U).
    \]
\end{proof}

\section{The canonical dimension of compactly induced representations}\label{sec:main}
Let $K\leq G$ be an open compact subgroup, then because $G$ is semi-simple and simply connected $K$ is contained in some maximal parahoric $K\subseteq P_x$ for some vertex $x\in\mathcal{B}_0$ (e.g. see \cite{kaletha}). We also let $\sigma$ be a representation of $K$ with open kernel. We will consider the representation $\pi=\cInd_K^G(\sigma)$ and we will write $V$ for the associated representation space.

Because the kernel of $\sigma$ is open and because the Moy-Prasad filtration of $P_x$ forms a neighbourhood basis of the identity in G, there must exist $R\in \R_{\geq 0}$ such that $P_{x,R}\leq \ker(\sigma)$.

\subsection{Reduction to the building using Mackey's theorem}
We will use the following version of Mackey's theorem.
\begin{definition}
Given two groups $H\leq E$, a representation $\phi$ of $H$ and an element $h\in E$, we write $\phi^{h}$ for the representation of $h^{-1}Hh$ with the same representation space as $\phi$ and with action given by $\phi^h(g)=\phi(hgh^{-1})$.
\end{definition}
\begin{theorem}[Mackey's theorem]
Let $E$ be a topological group and let ${H,K\leq E}$ be closed subgroups such that $K$ contains an open compact subgroup. Let $\tau_1$ be a finite-dimensional representation of $K$ and let $\tau_2$ be any representation of $H$ (neither of which need to be topological representations). Then there is a vector space isomorphism
\[
\Hom_{E}(\cInd_K^{E}(\tau_1),\Ind_H^{E}(\tau_2))\cong\prod_{HgK\in H\backslash E/K} \Hom_{K\cap g^{-1}Hg}(\tau_1,\tau_2^g),
\]
where we have fixed arbitrary representatives $g$ for the double cosets in $H\backslash E/K$ and where $\Ind$ denotes the induction functor. If moreover the image of $K$ in $H\backslash E$ is compact, then this isomorphism restricts to an isomorphism
\[
\Hom_{E}(\cInd_K^{E}(\tau_1),\cInd_H^{E}(\tau_2))\cong\bigoplus_{HgK\in H\backslash E/K} \Hom_{K\cap g^{-1}Hg}(\tau_1,\tau_2^g).
\]
\end{theorem}
\begin{proof}
See \cite{mackey}. Note that \cite{mackey} does not use smooth induction, however this is of no consequence since we additionally require $K$ to be open. It follows that the induction from $K$ is automatically smooth. We also have that homomorphisms map smooth vectors to smooth vectors, so the spaces of homomorphisms are the same independent of whether the induction from $H$ is smooth or not.
\end{proof}
\begin{corollary}\label{cor:mackey}
    Using the notation and assumptions from Mackey's Theorem, we have
    \[
    \Ind_H^E(\tau_2)^{K}\cong \prod_{HgK\in H\backslash E/ K} \tau_2^{gKg^{-1}\cap H}
    \]
    and if the image of $K$ in $H\backslash E$ is compact then we also have
    \[
    \cInd_H^E(\tau_2)^{K}\cong \bigoplus_{HgK\in H\backslash E/ K} \tau_2^{gKg^{-1}\cap H}
    \]
\end{corollary}
\begin{proof}
We have an isomorphism
\[
\Hom_{K}(\C_{\triv},\Ind_H^E(\tau_2))\xrightarrow{\sim} \Ind_H^E(\tau_2)^{K}
\]
given by $f\mapsto f(1)$, where $\C_{\triv}$ denotes the trivial representation of $K$. Because $K$ is open in $E$, we can use Frobenius reciprocity (see III.2.6.5 in \cite{renard}) to obtain
\[
\Hom_{K}(\C_{\triv},\Ind_H^E(\tau_2))\cong \Hom_E(\cInd_{K}^E(\C_{\triv}),\Ind_H^E(\tau_2)).
\]
We can now apply Mackey's theorem to conclude that
\begin{align*}
\Ind_H^E(\tau_2)^{K}&\cong \prod_{HgK\in H\backslash E/ K} \Hom_{K\cap g^{-1}Hg}(C_{\triv},\tau_2^g).
\end{align*}
Next we note that for $g\in E$ we have
\[
\Hom_{K\cap g^{-1}Hg}(C_{\triv},\tau_2^g)\cong \Hom_{gKg^{-1}\cap H}(C_{\triv},\tau_2)\cong \tau_2^{gKg^{-1}\cap H}.
\]
The first assertion is now clear and the second follows similarly.
\end{proof}
\begin{definition}
    Given $g\in G$ and $r\in \N$, we write 
    \[
    G_K(g,r)\coloneqq P_{gx,r}\cap K.
    \]
\end{definition}
\begin{lemma}\label{lem:vecIso}
Let $r\in\N$ and let $\Lambda(r)\subseteq G$ be a set of representatives for the double coset space $P_{x,r}\backslash G/K$, then we have a vector space isomorphism
\[
V^{P_{x,r}}\cong \bigoplus_{g\in\Lambda(r)}V_{\sigma}^{G_K(g^{-1},r)}.
\]
\end{lemma}
\begin{proof}
Let us note that 
\[
\Lambda^{-1}(r)\coloneqq\{g^{-1}\mid g\in\Lambda(r)\}
\]
is a set of representatives for $K\backslash G/P_{x,r}$. Now applying Corollary \ref{cor:mackey} we find 
\[
V^{P_{x,r}}\cong \bigoplus_{g\in \Lambda^{-1}(r)} V_\sigma^{gP_{x,r}g^{-1}\cap K}
\]
Next we have
\[
gP_{x,r}g^{-1}=P_{gx,r}
\]
and thus
\[
V^{P_{x,r}}\cong \bigoplus_{g\in \Lambda^{-1}(r)} V_\sigma^{G_K(g,r)}=\bigoplus_{g\in \Lambda(r)} V_\sigma^{G_K(g^{-1},r)}.
\]
\end{proof}
\begin{definition}
    Let $r\in \N$, then we write
    \[
    B(x,r;K,\sigma)\coloneqq \{g\in\Lambda(r)\mid V_\sigma^{G_K(g^{-1},r)}\neq 0\}.
    \]
\end{definition}
\begin{corollary}\label{cor:cdimEq}
    We have
    \[
    \cdim(\pi) = \log_q(|B(x,\bullet,\sigma)|).
    \]
\end{corollary}
\begin{proof}
Using Lemma \ref{lem:vecIso} we find that
\[
\dim(V^{P_{x,r}}) = \sum_{g\in\Lambda(r)} \dim\left(V_\sigma^{G_K(g^{-1},r)}\right) = \sum_{g\in\Lambda(r)} \dim\left(V_\sigma^{G_K(g^{-1},r)}\right).
\]
Clearly if $g\in\Lambda(r)$, then $1\leq \dim\left(V_\sigma^{G_K(g^{-1},r)}\right)\leq \dim(V_\sigma)$ and thus
\[
|B(x,r;K,\sigma)|\leq \dim(V^{P_{x,r}}) \leq \dim(V_\sigma)|B(x,r;K,\sigma)|.
\]
The result now follows easily from the definition of $\cdim(\pi)$.
\end{proof}
We have now reduced the problem of computing the canonical dimension of $\pi$ to determining the asymptotic behaviour of the cardinality of the set $B(x,r;K,\sigma)$.

\subsection{Lower bounds}
Consider the map 
\[
\iota:\Lambda(r)\to P_{x,r}\backslash \mathcal{B}_0^{\lambda(x)}
\]
given by
\[
g\mapsto P_{x,r}gx.
\]
Because $K\subseteq P_x$ and $G$ acts transitively on $\mathcal{B}_0^{\lambda(x)}$ with stabilizer $P_x$, the map $\iota$ is surjective and the corresponding map 
\[
P_{x,r}\backslash G/K\to P_{x,r}\backslash \mathcal{B}_0^{\lambda(x)}
\]
is independent of the choice of representatives $\Lambda(r)$.
\begin{lemma}\label{lem:estimate}
    Let $r\geq R+2$ be an integer, then we have
    \[
    P_{x,r}\backslash B(x,r-R)^{\lambda(x)}\subseteq \iota(B(x,r;K,\sigma)).
    \]
\end{lemma}
\begin{proof}
    Let $y\in B(x,r-R)^{\lambda(x)}$, then, because $\iota$ is surjective, we can find $g\in\Lambda(r)$ such that $y\in \iota(g)=P_{x,r} gx$. It now suffices to show that $g\in B(x,r;K,\sigma)$. Let $h\in P_{x,r}$ be such that $hgx=y$, then using the fact that $G$ acts isometrically we find 
    \[
    d(g^{-1}x,x)=d(x,gx)=d(hx,hgx)=d(x,y)\leq r-R.
    \]
    It now follows from Lemma \ref{lem:contain} that $P_{g^{-1}x,r}\subseteq P_{x,R}\subseteq \ker(\sigma)$. It is also clear that $G_K(g^{-1},r)=P_{g^{-1}x,r}$ and thus 
    \[
    V_\sigma^{G_K(g^{-1},r)}=V_\sigma\neq 0.
    \]
    By definition we now have $g\in B(x,r;K,\sigma)$, which completes the proof.
\end{proof}

\begin{lemma}\label{lem:lower}
    Let $r\in\N$, then we have
    \[
    |B(x,r;K,\sigma)|\geq \frac{1}{d+1}|B(o,r-R-2)|.
    \]
\end{lemma}
\begin{proof}
We have the following estimates
\begin{align*}
|B(x,r;K,\sigma)|&\geq |\iota(B(x,r;K,\sigma))|\\
&\geq |P_{x,r}\backslash B(x,r-R)^\lambda(x)|\\
&=|B(x,r-R)^{\lambda(x)}|\\
&\geq \frac{1}{d+1}|B(x,r-R-1)|.
\end{align*}
The first inequality is obvious, the second inequality follows by Lemma \ref{lem:estimate}, the equality follows by Lemma \ref{lem:contain} and the final inequality follows by Corollary \ref{cor:bound}. Now because the action of $G$ on $\mathcal{B}_0$ preserves distances and because every vertex is conjugate to either $o\in \A$ or to a vertex in $\A$ adjacent to $o$, we may assume that either $x=o$ or $x$ is adjacent to $o$. In both cases we find $B(o,r-R-2)\subseteq B(x,r-R-1)$ and thus we conclude
\[
|B(x,r;K,\sigma)|\geq \frac{1}{d+1}|B(o,r-R-2)|.
\]
\end{proof}
\begin{corollary}\label{cor:lower}
    We have
    \[
    \cdim(\pi)\geq \log_q(|B(o,\bullet)|).
    \]
\end{corollary}
\begin{proof}
    This follows upon combining Lemma \ref{lem:lower} with Corollary \ref{cor:cdimEq}.
\end{proof}
\begin{theorem}[Main theorem]\label{thm:fundamental}
    We have the following lower bounds for the canonical dimension of $\pi$, depending on the type of the root system $\Phi$:
    \[
        \begin{tabular}{c|c|c|c|c|c|c|}
           \cline{2-7}
            & $A_{2n}$ & $A_{2n+1}$ & $B_3$ & $B_{d\geq 4}$ & $C_d$ & $D_{d\geq 4}$\\ \hline
            \multicolumn{1}{|c|}{$\cdim(\pi)\geq$} & $n(n+1)$ & $(n+1)^2$ & $5$ & $\frac{d^2}{2}$ & $\frac{d(d+1)}{2}$ & $\frac{d(d-1)}{2}$ \\\hline
        \end{tabular}
    \]
    \[
        \begin{tabular}{c|c|c|c|c|c|}
           \cline{2-6}
            & $E_6$ & $E_7$ & $E_8$ & $F_4$ & $G_2$\\ \hline
            \multicolumn{1}{|c|}{$\cdim(\pi)\geq$} & $16$ & $27$ & $46$ & $11$ & $4$ \\\hline
        \end{tabular},
    \]
    where in type $A$ we have $d=2n$ or $d=2n+1$ for some $n\in\N$ depending on the parity of $d$.
\end{theorem}
\begin{proof}
    This follows from Corollary \ref{cor:lower} and Theorem \ref{thm:cardinality}. For the $G_2$ entry, we have also used the fact that the canonical dimension is always an integer.
\end{proof}
\begin{remark}
    We have allowed for a general coefficient field $\mathcal{F}$, however it should be said that the content of this theorem is not very interesting in the case where $\mathcal{F}$ has characteristic $p$. This is because in characteristic $p$ any smooth representation of a pro-$p$ group has nontrivial fixed vectors. Combining this with Lemma \ref{lem:vecIso} and the fact that $\Lambda(r)$ is infinite we see that $\cdim(\pi)=\infty$ regardless of $K$ or $\sigma$. In other words, in characteristic $p$ every representation compactly induced from an open compact subgroup is non-admissible. 
\end{remark}
An immediate corollary of our main theorem is the following important result.
\begin{theorem}\label{thm:corollary}
    Assume one of the following:
    \begin{enumerate}[(i)]
        \item $\mathcal{F}=\C$ and $p$ does not divide the order of the Weyl group associated with the root system $\Phi$,
        \item $p$ is odd, $\mathcal{F}$ has characteristic unequal to $p$ and $G$ is of type $B,C$ or $D$.
    \end{enumerate}
    Let $\pi$ be an irreducible supercuspidal representation of $G$. Then the lower bounds of Theorem \ref{thm:fundamental} apply to $\pi$.
\end{theorem}
\begin{proof}
    It suffices to show that $\pi$ is (isomorphic to a representation) compactly induced from an open compact subgroup and that the representation $\sigma$ being induced has open kernel, because then Theorem \ref{thm:fundamental} applies. If $\sigma$ is smooth and finite-dimensional then it has open kernel. Under assumption $(i)$ it follows from \cite{exhaustion} that $\pi$ is obtained from a process called Yu's construction. Examining Yu's construction as detailed for example in \cite{construction} or in the original paper \cite{yu} shows that $\pi$ is of the desired form. Under assumption $(ii)$ it follows from \cite{stevens} that $\pi$ is of the desired form.
\end{proof}
\begin{remark}
    It is expected that compact induction from an open compact subgroup plays an important role in the construction of supercuspidal representations even without the hypotheses of this theorem. Therefore it is likely that Theorem \ref{thm:fundamental} will have implications for supercuspidal representations even in a more general setting.
\end{remark}

\section{The local character expansion}\label{sec:local}
For the rest of this paper we will specialize to the case where $\mathcal{F}=\C$. Ìn this case there are more tools at our disposal, in particular the Harish-Chandra-Howe local character expansion. We will use these tools to establish some facts about the canonical dimension. We will also use them to define the wavefront set and explain the relationship between the canonical dimension and the wavefront set. In this section we recall the local character expansion.

Let us write $C^\infty_c(G)$ for the space  of complex valued, compactly supported, locally constant functions on $G$, where $G$ is equipped with the $p$-adic topology. This space can be equipped with a convolution product, turning it into an algebra known as the \emph{Hecke algebra} of $G$. Its linear dual $C_c^\infty(G)^\prime$ is the space of \emph{distributions} on $G$. Given an admissible representation $(\pi,V)$ of $G$, we can define an action of the Hecke algebra $\pi:C^\infty_c(G)\to \End(V)$, where $C^\infty_c(G)$ acts by finite rank operators. We can then define a \emph{distribution character} $\Theta_\pi\in C^\infty_c(G)^\prime$ by
\[
\Theta_\pi(f) = \tr(\pi(f))
\]
for all $f\in C^\infty_c(G)$.

Similarly, on the Lie algebra $\mathfrak{g}$ we consider $C_c^\infty(\mathfrak{g})$ and $C_c^\infty(\mathfrak{g})^\prime$. We write $O(0)$ for the set of nilpotent orbits in $\mathfrak{g}$. Given $O\in O(0)$, Deligne and Rao \cite{rao} showed that there exists a well-defined $G$-invariant distribution $\mu_O\in C^\infty_c(\mathfrak{g})^\prime$ which essentially integrates along $O$. This distribution is called a \emph{ nilpotent orbital integral} and it is unique up to multiplication by a positive scalar. We fix choices for these $\mu_O$.

One important result about nilpotent orbital integrals is that they are \emph{homogeneous} in some sense. This is not unexpected once we realize that nilpotent orbits $O\in O(0)$ are homogeneous in the sense that $tO=O$ for all $t\in F^\times$.
\begin{definition}
    Let $t\in F^\times$ and let $f\in C_c^\infty(\mathfrak{g})$, we define a function $f_t$ by
    \[
    f_t(X)=f(t^{-1}X).
    \]
    Now let $T\in C_c^\infty(\mathfrak{g})^\prime$ be a distribution, then we define a distribution $T_t$ by
    \[
    T_t(f) = T(f_t),
    \]
    for all $f\in C_c^\infty(\mathfrak{g})$.
\end{definition}
\begin{proposition}\label{prop:hom}
Given $t\in F^\times$ and $O\in O(0)$, we have
\[
(\mu_O)_{t^2} = |t|^{\dim(O)}\mu_O,
\]
where $\dim(O)$ is the dimension of $O$.
\end{proposition}
\begin{proof}
    This is Lemma 3.2 in \cite{harish}.
\end{proof}

\subsection{Lattices and the Fourier transform}
An open compact $\mathcal{O}$-submodule of $\mathfrak{g}$ is called a \emph{lattice}. Similarly to the filtrations of the parahoric subgroups, Moy and Prasad (\cite{moy-prasad}) defined filtrations $\mathfrak{g}_{x,r}$ of lattices in $\mathfrak{g}$. Here $x\in\mathcal{B}$ and $r\in \R$. These lattices have the property that for $r$ large enough we have $\exp(\mathfrak{g}_{x,r})=P_{x,r}$. Here $\exp$ is the exponential map into $G$, which is defined on a neighbourhood of the origin in $\mathfrak{g}$. Also, for any $r\in \R$ and $n\in \N$ we have $\mathfrak{g}_{x,r+n}=\varpi^n\mathfrak{g}_{x,r}$. It follows that $\{\mathfrak{g}_{x,n}\}_{n\in\N}$ is neighbourhood basis of the origin. Moreover, the following lemma shows that we can normalize the Haar measure on $G$ in a way that is compatible with the Haar measure on $\mathfrak{g}$.
\begin{lemma}\label{lem:measure}
    Given $x\in\mathcal{B}$ and a Haar measure $dX$ on $\mathfrak{g}$, we can normalize the Haar measure $dg$ on $G$ such that 
    \[
    dX(\mathfrak{g}_{x,r}) = dg(P_{x,r})
    \]
    for all $r\geq 1$.
\end{lemma}
\begin{proof}
    It follows from \cite{moy-prasad} that if $r>s>0$, then $P_{x,s}/P_{x,r}\cong \mathfrak{g}_{x,s}/\mathfrak{g}_{x,r}$ and both are finite sets. If we normalize $dg$ such that $dX(\mathfrak{g}_{x,1}) = dg(P_{x,1})$ it follows that
    \[    |\mathfrak{g}_{x,1}/\mathfrak{g}_{x,r}|dX(\mathfrak{g}_{x,r})=dX(\mathfrak{g}_{x,1}) = dg(P_{x,1})=|P_{x,1}/P_{x,r}|dg(P_{x,r}).
    \]
    Now the isomorphism $P_{x,s}/P_{x,r}\cong \mathfrak{g}_{x,s}/\mathfrak{g}_{x,r}$ shows that the two cardinalities are equal and the result follows.    
\end{proof}

For the rest of this text, we fix a symmetric, non-degenerate, $G$-invariant bilinear form $B$ on $\mathfrak{g}$.
\begin{definition}
    Let $L\subseteq\mathfrak{g}$ be a lattice, then we define its dual $L^*$ by
\[
L^*\coloneqq\{X\in\mathfrak{g}\mid \forall Y\in L,\, B(X,Y)\in \mathfrak{p}\}.
\]
This is again a lattice.
\end{definition}

\begin{definition}
    Let $\Omega\subseteq\mathfrak{g}$ be a subset, then we write $\mathbbm{1}_\Omega$ for the characteristic function on $\Omega$.
\end{definition}
\begin{proposition}\label{prop:homlat}
    Let $L\subseteq\mathfrak{g}$ be a lattice and let $t\in F^\times$, then 
    \[
    (\mathbbm{1}_L)_t=\mathbbm{1}_{tL}\text{ and }(tL)^*=t^{-1}L^*.
    \]
\end{proposition}
\begin{proof}
This is clear.
\end{proof}

We now turn our attention to defining the Fourier transform. We fix a complex valued character $\Lambda$ of the additive group $F$ which is nontrivial on $\mathcal{O}$, but is trivial on $\mathfrak{p}$.
\begin{definition}
    Let $f\in C_c^\infty(\mathfrak{g})$ and define
    \[
    \widehat{f}(Y) \coloneqq \int_\mathfrak{g} f(X)\Lambda(B(X,Y))dX,
    \]
    for all $Y\in\mathfrak{g}$. Here $dX$ is a Haar measure on $\mathfrak{g}$. Then $\widehat{f}\in C_c^\infty(\mathfrak{g})$ is called the \emph{Fourier transform} of $f$. We can normalize $dX$ such that $\hat{\hat{f}}(X)=f(-X)$ for all $f\in C^\infty_c(\mathfrak{g})$ and we will choose this normalization.

    Let $T\in C^\infty_c(\mathfrak{g})^\prime$, then the \emph{Fourier transform} of $T$ is defined by
    \[
    \widehat{T}(f)\coloneqq T(\widehat{f})
    \]
    for all $f\in C_c^\infty(\mathfrak{g})$ and this defines a distribution $\widehat{T}\in C_c^\infty(\mathfrak{g})^\prime$.
\end{definition}
\begin{proposition}
    Let $O\in O(0)$, then the Fourier transform $\widehat{\mu_O}$ is represented by a locally integrable function also denoted $\widehat{\mu_{O}}$.
\end{proposition}
\begin{proof}
    See Theorem 4.4 in \cite{harish}.
\end{proof}
\begin{proposition}\label{prop:fourierlat}
    Let $L\subseteq\mathfrak{g}$ be a lattice, then 
    \[
    \widehat{\mathbbm{1}_L}=dX(L)\mathbbm{1}_{L^*}.
    \]
\end{proposition}
\begin{proof}
Let $Y\in\mathfrak{g}\setminus L^*$. We will show that $\widehat{\mathbbm{1}_L}(Y)=0$. Write $L^\prime$ for the kernel of the group homomorphism $\Lambda(B(Y,\cdot)):L\to \C^\times$. Because $L$ is compact and $B(Y,\cdot)$ is continuous, $B(Y,L)$ is compact. Now since $\mathfrak{p}$ is open we have that $B(Y,L)/\mathfrak{p}$ is finite and since $\mathfrak{p}$ is in the kernel of $\Lambda$, this implies that $\Lambda(B(Y,L))\cong L/L^\prime$ is finite. Since $\Lambda(B(Y,L))$ is a finite subgroup of the group of units of a field, it is cyclic. Let $\Lambda(B(Y,L))$ be generated by $\zeta$ and let $l$ be its order. 

To prove that $l>1$, let $X\in L$ such that $B(Y,X)\notin\mathfrak{p}$. Note that such a $X$ must exist by our assumption that $Y\notin L^*$. Now for some $n\geq 0$ we have $B(Y,\varpi^nX)\in \mathcal{O}^\times$. Write ${\alpha\coloneqq B(Y,\varpi^nX)^{-1}}$, then $X^\prime\coloneqq\alpha\varpi^nX\in L$ and $B(Y,X^\prime)=1$. Since $\Lambda$ is nontrivial on $\mathcal{O}$, we have $\Lambda(B(Y,X^\prime))\neq 1$ and hence $\Lambda(B(Y,X^\prime))$ defines a nontrivial element of $\Lambda(B(Y,L))$, proving that $l>1$.

Now since $\Lambda(B(Y,\cdot))$ is constant on the cosets of $L^\prime$ in $L$, we have
\begin{align*}
\widehat{\mathbbm{1}_L}(Y) &= \int_\mathfrak{g} \mathbbm{1}_L(X)\Lambda(B(Y,X))dX\\
&=\int_L \Lambda(B(Y,X))dX\\
&=\sum_{Z\in L/L^\prime} dX(L^\prime)\Lambda(B(Y,Z))\\
&=dX(L^\prime)\sum_{n=0}^{l-1}\zeta^n=0.
\end{align*}

Conversely, if $Y\in L^*$, then $B(Y,L)\subseteq\mathfrak{p}$ and $\Lambda(B(Y,L))=\{1\}$, so with the notation above we have $L=L^\prime$ and $l=1$. Therefore a similar computation shows
\[
\widehat{\mathbbm{1}_L}(Y) = dX(L).
\]
We conclude that $\widehat{\mathbbm{1}_L}=dX(L)\mathbbm{1}_{L^*}$.
\end{proof}

\subsection{The local character expansion and the wavefront set}
We can use the exponential map to pull the distribution character of a representation back to a map defined near the origin of the Lie algebra $\mathfrak{g}$. The local character expansion expands this pull-back as follows.
\begin{theorem}
Let $\pi$ be an irreducible smooth representation of $G$, then there exist unique complex numbers $c_O(\pi)$ indexed by $O\in O(0)$ such that
\[
\Theta_\pi(\exp(X)) = \sum_{O\in O(0)}c_{O}(\pi)\widehat{\mu_{O}}(X),
\]
for all $X\in\mathfrak{g}$ sufficiently close to zero. 
\end{theorem}
\begin{proof}
    See Theorem 16.2 in \cite{harish} with $\gamma=0$.
\end{proof}
\begin{remark}
    It is known that an irreducible smooth representation is admissible and hence the statement of this theorem makes sense.
\end{remark}
\begin{remark}
    While the theorem is formulated only for irreducible representations, it can be shown that the distribution character is additive on the Grothendieck group of admissible representations and hence we can extend the local character expansion to any finite length representation.
\end{remark}
We will later use this local character expansion to show that the canonical dimension of a representation is related to some nilpotent orbit. We will now use it to define the wavefront set of a representation.
\begin{definition}
    We define a partial order on $O(0)$ by $O\leq O^\prime$ if and only if $O$ is contained in the closure of $O^\prime$ with respect to the analytic topology.
\end{definition}
\begin{definition}
    Let $\pi$ be an admissible representation of $G$ of finite length, then we define the wavefront set $\WF(\pi)$ by
    \[
    \WF(\pi)\coloneqq\max\{O\in O(0)\mid c_O(\pi)\neq 0\},
    \]
    where the maximum is taken with respect to the partial order on $O(0)$. Note that this set need not be a singleton.
\end{definition}

\section{The canonical dimension of complex representations}\label{sec:complexcdim}
In this section we continue with our assumption that $\mathcal{F}=\C$. We use the local character expansion to establish some basic facts about the canonical dimension, which were known before. We also introduce the wavefront set and describe the relationship between the canonical dimension and the wavefront set.
\begin{proposition}\label{prop:canDimExaSeq}
    Fix $x\in\mathcal{B}$ and let 
    \[
    0\to V_1\to V_2\to V_3\to 0
    \]
    be an exact sequence of admissible representations of $G$, then 
    \[
    \cdim(V_2)=\max\{\cdim(V_1),\cdim(V_3)\},
    \]
    where the canonical dimensions are computed with respect to $x$.
\end{proposition}
\begin{proof}
    This follows easily from the fact that given an open compact subgroup ${K\leq G}$, the functor $V\mapsto V^K$ from smooth representations of $G$ to vector spaces is exact (see Proposition III.1.5 in \cite{renard}).
\end{proof}
It has long been known that the canonical dimension is related to nilpotent orbits via the local character expansion (see for example Section 16 in \cite{vigneras} or Subsection 5.1 in \cite{barbaschmoy}). However, we find it useful to spell out this relationship explicitly and in detail.
\begin{theorem}\label{thm:canDim}
Let $(\pi,V)$ be an irreducible admissible representation of $G$ and let $O\in O(0)$ be an orbit of maximal dimension subject to the condition $c_O(\pi)\neq 0$. Then for any $x\in\mathcal{B}$ we have
\[
\cdim(\pi) = \lim_{n\to\infty} \frac{\log_q(\dim(V^{P_{x,n}}))}{n} = \frac{\dim(O)}{2}.
\]
\end{theorem}
\begin{proof}
Fix $x\in\mathcal{B}$ and use the normalization of $dg$ furnished by Lemma \ref{lem:measure}. Because $\dim(V^{P_{x,n}})$ is monotone increasing in $n$, the limit
\[
\lim_{n\to\infty} \frac{\log_q(\dim(V^{P_{x,n}}))}{n}
\]
exists if and only if
\[
\lim_{n\to\infty} \frac{\log_q(\dim(V^{P_{x,2n}}))}{2n}
\]
exists, in which case they are equal.

Let $e_n\coloneqq \frac{1}{dg(P_{x,n})}\mathbbm{1}_{P_{x,n}}\in C^\infty_c(G)$, then $\pi(e_n)$ is a projector onto $V^{P_{x,n}}$. This is a finite rank operator whose trace equals $\dim(V^{P_{x,n}})$ and thus
\[
\dim(V^{P_{x,n}}) = \tr(\pi(e_n)) = \Theta_\pi(e_n).
\]
Because $\exp(\mathfrak{g}_{x,n})=P_{x,n}$ for all $n$ and $\{\mathfrak{g}_{x,n}\}_{n\in\N}$ forms a neighbourhood basis of zero, we can apply the local character expansion to obtain that for $n$ sufficiently large
\[
\dim(V^{P_{x,n}}) = \sum_{O^\prime\in O(0)} c_{O^\prime}(\pi)\widehat{\mu_{O^\prime}}(e_n\circ\exp).
\]
Now note that $e_n\circ\exp=\frac{1}{dg(P_{x,n})}\mathbbm{1}_{\mathfrak{g}_{x,n}}$, then it follows that for all $n$ sufficiently large
\begin{align*}
\dim(V^{P_{x,n}}) &= \frac{1}{dg(P_{x,n})}\sum_{O^\prime\in O(0)} c_{O^\prime}(\pi)\widehat{\mu_{O^\prime}}(\mathbbm{1}_{\mathfrak{g}_{x,n}})\\
&=\frac{1}{dg(P_{x,n})}\sum_{O^\prime\in O(0)} c_{O^\prime}(\pi)\mu_{O^\prime}(\widehat{\mathbbm{1}_{\mathfrak{g}_{x,n}}})\\
&=\frac{dX(\mathfrak{g}_{x,n})}{dg(P_{x,n})}\sum_{O^\prime\in O(0)} c_{O^\prime}(\pi)\mu_{O^\prime}(\mathbbm{1}_{\mathfrak{g}_{x,n}^*}),
\end{align*}
where the last equality follows from Proposition \ref{prop:fourierlat}. Our normalization of the Haar measure on $G$ guarantees that $\frac{dX(\mathfrak{g}_{x,n})}{dg(P_{x,n})}=1$ for all $n\geq 1$. Next, for sufficiently large $n$, we compute
\begin{align*}
\dim(V^{P_{x,n}})&=\sum_{O^\prime\in O(0)} c_{O^\prime}(\pi)\mu_{O^\prime}(\mathbbm{1}_{\mathfrak{g}_{x,n}^*})\\
&=\sum_{O^\prime\in O(0)} c_{O^\prime}(\pi)\mu_{O^\prime}(\mathbbm{1}_{(\varpi^{n}\mathfrak{g}_{x,0})^*})\\
&=\sum_{O^\prime\in O(0)} c_{O^\prime}(\pi)\mu_{O^\prime}(\mathbbm{1}_{\varpi^{-n}\mathfrak{g}_{x,0}^*})\\
&=\sum_{O^\prime\in O(0)} c_{O^\prime}(\pi)\mu_{O^\prime}((\mathbbm{1}_{\mathfrak{g}_{x,0}^*})_{\varpi^{-n}})\\
&=\sum_{O^\prime\in O(0)} c_{O^\prime}(\pi)(\mu_{O^\prime})_{\varpi^{-n}}(\mathbbm{1}_{\mathfrak{g}_{x,0}^*}),
\end{align*}
where the third and fourth equalities follow from Proposition \ref{prop:homlat}. By Proposition \ref{prop:hom}, we obtain that for $n$ sufficiently large
\begin{align*}
\dim(V^{P_{x,2n}})&=\sum_{O^\prime\in O(0)} c_{O^\prime}(\pi)|\varpi^{-n}|^{\dim(O^\prime)}\mu_{O^\prime}(\mathbbm{1}_{\mathfrak{g}_{x,0}^*})\\
&=\sum_{O^\prime\in O(0)} c_{O^\prime}(\pi)q^{n\dim(O^\prime)}\mu_{O^\prime}(\mathbbm{1}_{\mathfrak{g}_{x,0}^*}).
\end{align*}
Now using that $\mathfrak{g}^*_{x,0}$ is a lattice and hence an open neighbourhood of $0\in\mathfrak{g}$, it is clear that $\mu_{O^\prime}(\mathbbm{1}_{\mathfrak{g}_{x,0}^*})\neq 0$ for all $O^\prime\in O(0)$. 

Next we compute
\begin{align*}
\log_q(\dim(V^{P_{x,2\bullet}})) &= \max_{O^\prime\in O(0)}\log_q\left(c_{O^\prime}(\pi)q^{\bullet\dim(O^\prime)}\mu_{O^\prime}(\mathbbm{1}_{\mathfrak{g}_{x,0}^*})\right)\\
&=\max_{O^\prime\in O(0),\: c_{O^\prime}(\pi)\neq 0}\log_q\left(q^{\bullet\dim(O^\prime)}\right)\\
&=\max\{\dim(O^\prime)\mid O^\prime\in O(0),\:c_{O^\prime}(\pi)\neq 0\}
\end{align*}
and in fact all of these limsups are limits. If we now fix an $O\in O(0)$ which maximizes $\dim(O^\prime)$ among those nilpotent orbits with $c_{O^\prime}(\pi)\neq 0$, then it is clear that
\[
\cdim(\pi) = \frac{1}{2}\log_q(\dim(V^{P_{x,2\bullet}}))=\frac{\dim(O)}{2} = \lim_{n\to\infty} \frac{\log_q(\dim(V^{P_{x,n}}))}{n}
\]
and this is independent of our choice of $x\in\mathcal{B}$.
\end{proof}
\begin{corollary}\label{cor:canDim}
    Let $(\pi,V)$ be an admissible representation of $G$ of finite length, then there exists ${O\in O(0)}$ such that for any $x\in\mathcal{B}$
    \[
    \cdim(\pi) = \lim_{n\to\infty} \frac{\log_q(\dim(V^{P_{x,n}}))}{n} = \frac{\dim(O)}{2}.
    \]
    In particular, $\cdim(\pi)$ is a finite integer, independent of the choice of $x\in\mathcal{B}$.
\end{corollary}
\begin{proof}
We prove the first statement first. Because the representation $\pi$ has finite length we can use induction on the length. The base case is covered by Theorem \ref{thm:canDim} and the induction step follows straightforwardly from Proposition \ref{prop:canDimExaSeq}. The other statements now follow easily, keeping in mind that nilpotent orbits have even dimension.
\end{proof}
\begin{remark}
    It is known that any smooth representation of $G$ of finite length is admissible and moreover, an admissible representation of $G$ has finite length if and only if it is finitely generated.
\end{remark}

\subsection{The canonical dimension and the wavefront set}
We can now relate the canonical dimension to the wavefront set. Let $\pi$ be an admissible representation of $G$ of finite length and let $O\in O(0)$ be an orbit provided by Corollary \ref{cor:canDim}. The proof of Theorem \ref{thm:canDim} shows that $O$ has maximal dimension among those orbits $O^\prime\in O(0)$ satisfying $c_{O^\prime}(\pi)\neq 0$. Then $O$ is also maximal with respect to the poset structure on $O(0)$ and thus $O\in \WF(\pi)$. Moreover, $\WF(\pi)$ cannot contain an orbit of greater dimension. We see that $\cdim(\pi)=\dim(\WF(\pi))$, where we have written $\dim(\WF(\pi))$ for the dimension of the variety obtained by taking the union of all elements in $\WF(\pi)$. It is now clear that the canonical dimension of $\pi$ can be deduced from its wavefront set. The converse is false, however the canonical dimension does tell us about the dimension of the wavefront set. In this sense a bound for the canonical dimension gives us a bound for the wavefront set.

\section{Complex depth-zero supercuspidal representations}\label{sec:supercuspidal}
We continue with our assumption that $\mathcal{F}=\C$. In this section we prove an upper bound for the canonical dimension of an irreducible depth-zero supercuspidal representation of $G$. We fix such a representation $\pi$. It follows that there exists some vertex $x\in\mathcal{B}_0$ such that $\pi$ is of the form $\cInd_{P_x}^G(\widetilde{\sigma})$, where $\widetilde{\sigma}$ is the inflation to $P_x$ of a cuspidal representation $\sigma$ of $\mathcal{G}_x$ (see Proposition 6.8 in \cite{moy-prasad2}). 

We use the notation from Section \ref{sec:main}. Note that we have $K=P_x$ and therefore the map $\iota: \Lambda(r)\to P_{x,r}\backslash \mathcal{B}_0^{\lambda(x)}$ is injective for every $r\in\N$ and hence $\iota$ is a bijection. 
\begin{proposition}\label{prop:upContain}
    Let $r\in\N$, then we have
    \[
   \iota( B(x,r;P_{x},\widetilde{\sigma}))\subseteq P_{x,r}\backslash B\left(x,1+(r+1)\sum_{i=1}^d c_i\right).
    \]
\end{proposition}
\begin{proof}
Let $P_{x,r}y\in P_{x,r}\mathcal{B}_0^{\lambda(x)}$ such that $d(x,y)> 1+(r+1)\sum_{i=1}^d c_i$. Note that $d(x,y)$ is independent of the chosen representative $y$ of $P_{x,r}y$. Let $g\coloneqq \iota^{-1}(P_{x,r}y)$, then we have to show that $V_{\widetilde{\sigma}}^{G_K(g^{-1},r)}=0$. Because $\sigma$ is cuspidal, it suffices to show that the image of $G_K(g^{-1},r)=P_{g^{-1}x,r}\cap P_x$ in $\mathcal{G}_x$ contains the unipotent radical of a parabolic. Note that $d(g^{-1}x,x)=d(x,gx)=d(x,y)>1+(r+1)\sum_{i=1}^d c_i$.

We may assume without loss of generality that $x=o$ or else $x$ is adjacent to $o$. In either case, we conclude $d(o,g^{-1}x)\geq d(x,g^{-1}x)-1>(r+1)\sum_{i=1}^d c_i$. We may furthermore assume by conjugating that $g^{-1}x\in C^+$. From this it follows by Lemma \ref{lem:polytope} that $\alpha_0(g^{-1}x)> (r+1)\sum_{i=1}^d c_i$. Now recall that
\[
\alpha_0(g^{-1}x)=\sum_{i=1}^d c_i\alpha_i(g^{-1}x)
\]
and thus it follows that there exists some $k$ such that $\alpha_k(g^{-1}x)>r+1$.

It follows from \cite{bruhat1} (see Proposition 6.4.23) that $\mathcal{G}_x$ is a reductive group of type $\Phi_x$, where $\Phi_x$ is defined by
\[
\Phi_x=\{\alpha\in\Phi\mid \alpha(x)\in\Z\}.
\]
Because $x$ is a vertex, the rank of $\Phi_x$ is $d$. The root system $\Phi_x$ inherits a system of positive roots from $\Phi$ via $\Phi_x^+=\Phi_x\cap\Phi^+$. Let $\{\beta_1,\dots,\beta_d\}$ be the corresponding set of fundamental roots. There exists $j$ such that when we express $\beta_j$ as a nonnegative integral linear combination of the $\alpha_i$, the $\alpha_k$ appears with positive coefficient. It follows that $\beta_j(g^{-1}x)\geq \alpha_k(g^{-1}x)>r+1$. Let $\Phi_x^+$ be ordered in the usual way, then $(f_{g^{-1}x}+r)(\beta)<-1$ for all $\beta_j\leq\beta\in\Phi^+$. Now because we assumed that $x$ is either equal or adjacent to $o$ we have $(f^*_x)(\beta)\geq 0$ for all $\beta\in\Phi^+$. We conclude that the projection of $P_{g^{-1}x,r}\cap P_x$ in $\mathcal{G}_x$ contains all the root subgroups corresponding to roots $\beta\in\Phi_x^+$ satisfying $\beta\geq \beta_j$. The group generated by these root subgroups is precisely the unipotent radical of the parabolic corresponding to the subset $\{\beta_1,\dots,\beta_{j-1},\beta_{j+1},\dots,\beta_d\}$ of fundamental roots. Thus the projection of $P_{g^{-1}x,r}\cap P_x$ in $\mathcal{G}_x$ contains the unipotent radical of a parabolic, which finishes the proof.
\end{proof}
\begin{lemma}\label{lem:containUpperBound}
    Let $r\in\N$, then we have
    \[
    \left|P_{x,r}\backslash B\left(x,1+(r+1)\sum_{i=1}^d c_i\right)\right|\leq \left|P_{o,r+1}\backslash B\left(o,2+(r+1)\sum_{i=1}^d c_i\right)\right|.
    \]
\end{lemma}
\begin{proof}
    The action of $G$ on $\mathcal{B}_0$ preserves distances and every vertex is conjugate either to the origin $o\in\A$ or to a vertex adjacent to $o$. In either case we have that 
    \[
    B\left(x,1+(r+1)\sum_{i=1}^d c_i\right)\subseteq B\left(o,2+(r+1)\sum_{i=1}^d c_i\right).
    \]
    Moreover, $P_{o,r+1}\subseteq P_{x,r}$ and thus the result follows.
\end{proof}
\begin{corollary}\label{cor:upperBound}
    Let $r\in\N$, then we have
    \[
    |B(x,r;P_{x},\widetilde{\sigma})|\leq \gamma\Gamma \left(2+(r+1)\sum_{i=1}^d c_i\right)^d q^{(r+1)|\Phi^+|}.
    \]
\end{corollary}
\begin{proof}
    This follows immediately upon combining Proposition \ref{prop:upContain}, Lemma \ref{lem:containUpperBound}, Corollary \ref{cor:upBound} and the fact that $\iota$ is a bijection. 
\end{proof}
\begin{corollary}\label{cor:complexUpperBound}
    We have
    \[
    \cdim(\pi)\leq |\Phi^+|.
    \]
\end{corollary}
\begin{proof}
This is immediate when combining Corollary \ref{cor:upperBound} with Corollary \ref{cor:cdimEq}.    
\end{proof}
\begin{remark}
    This upper bound agrees with a well-known and trivial upper bound for the wavefront set and equals half the dimension of the principal nilpotent orbit of $G$. This result could also easily have been deduced from the relationship between the canonical dimension and the wavefront set, however our methods avoid making use of the local character expansion, which is quite a deep result.
\end{remark}
\begin{remark}
It would be natural to try to extend this upper bound to arbitrary irreducible depth-zero representations of $G$. Any such representation is a subquotient of a parabolic induction of an irreducible depth-zero supercuspidal representation. It follows from Theorem \ref{thm:parabolic} that our upper bound is preserved under parabolic induction and moreover Proposition \ref{prop:canDimExaSeq} implies that it is preserved under taking subquotients as well. So it would seem that we can extend our upper bound without much difficulty. However, in this paper we have restricted our attention to the class of split, simply connected, absolutely almost simple algebraic groups and this class is not closed under taking Levi subgroups. It follows that the argument outlined above does not immediately apply. However, we expect our results to be valid for the much larger class of all split reductive groups which means the argument would go through.
\end{remark}

\bibliographystyle{plainmath}
\bibliography{references.bib}

\end{document}